\documentclass[11pt]{article}

\usepackage[T1]{fontenc}
\usepackage[latin9]{inputenc}
\usepackage{mathtools}

\usepackage[numbers]{natbib}
\usepackage{amsthm}
\usepackage{amsmath}
\usepackage{amssymb}

\usepackage{graphicx}

\usepackage{fullpage}
\makeatletter
\theoremstyle{plain}
\newtheorem{thm}{\protect\theoremname}
  \theoremstyle{plain}
  \newtheorem{cor}{\protect\corollaryname}
  \theoremstyle{plain}
  \newtheorem{lem}{\protect\lemmaname}

\DeclareMathOperator*{\argmin}{argmin}

\DeclareMathAlphabet{\mathcal}{OMS}{cmsy}{m}{n}

\@ifundefined{showcaptionsetup}{}{%
 \PassOptionsToPackage{caption=false}{subfig}}
\usepackage{subfig}
\makeatother

\usepackage[english]{babel}
  \providecommand{\lemmaname}{Lemma}
\providecommand{\corollaryname}{Corollary}
\providecommand{\theoremname}{Theorem}

\begin{document}
\global\long\def\mb#1{\boldsymbol{#1}}
\global\long\def\mbb#1{\mathbb{#1}}
\global\long\def\mc#1{\mathcal{#1}}
\global\long\def\mcc#1{\mathscr{#1}}
\global\long\def\mr#1{\mathrm{#1}}
\global\long\def\msf#1{\mathsf{#1}}
\global\long\def\E{\mb{\msf E}}
\global\long\def\P{\mb{\msf P}}
\global\long\def\tbullet{\tiny{\mbox{\ensuremath{\bullet}}}}

\title{Near-Optimal Estimation of Simultaneously Sparse and Low-Rank Matrices from Nested Linear Measurements}

\author{{\sc Sohail Bahmani},\\[2pt]
School of Electrical and Computer Engineering, Georgia Institute of Technology, Atlanta, GA\\
\texttt{sohail.bahmani@ece.gatech.edu}\\[2pt]
{\sc Justin Romberg}\\[2pt]
School of Electrical and Computer Engineering, Georgia Institute of Technology, Atlanta, GA\\
\texttt{jrom@ece.gatech.edu}}

\maketitle

\begin{abstract}
In this paper we consider the problem of estimating simultaneously
low-rank and row-wise sparse matrices from nested linear measurements
where the linear operator consists of the product of a linear operator
$\mc W$ and a matrix $\mb{\varPsi}$. Leveraging the nested structure
of the measurement operator, we propose a computationally efficient
two-stage algorithm for estimating the simultaneously structured target
matrix. Assuming that $\mc W$ is a restricted isometry for low-rank
matrices and $\mb{\varPsi}$ is a restricted isometry for row-wise
sparse matrices, we establish an accuracy guarantee that holds uniformly
for all sufficiently low-rank and row-wise sparse matrices with high
probability. Furthermore, using standard tools from information theory,
we establish a minimax lower bound for estimation of simultaneously
low-rank and row-wise sparse matrices from linear measurements that
need not be nested. The accuracy bounds established for the algorithm,
that also serve as a minimax upper bound, differ from the derived
minimax lower bound merely by a polylogarithmic factor of the dimensions.
Therefore, the proposed algorithm is nearly minimax optimal. We also
discuss some applications of the proposed observation model and evaluate
our algorithm through numerical simulation.
\end{abstract}
\section{\label{sec:Introduction}Introduction}

In this paper we study the problem of estimating a matrix $\mb X^{\star}$
that is simultaneously low-rank and row-wise sparse from the noisy
linear measurements 
\begin{align}
\mb y & =\mc A\left(\mb X^{\star}\right)+\mb z,\label{eq:measurements}
\end{align}
 where $\mc A$ is a known linear operator and $\mb z$ models the
additive error or noise. In particular, we consider this problem in
the high-dimensional regime where the dimension of the measurements
$\mb y$ can be much smaller than the ambient dimension of the target
$\mb X^{\star}$. Estimation of low-rank matrices from the linear
measurements (\ref{eq:measurements}) has been studied extensively
in various settings (see, e.g., \cite{Candes_Exact_2009}, \cite{Recht_Guaranteed_2010},
and \cite{Koltchinskii_Nuclear_2011}). However, in cases where the
target matrix is characterized by two or more structures simultaneously,
the research regarding efficient and statistically optimal estimation
is less developed. Ideally, different structures of the target matrix
can be exploited to improve the sample complexity.

In this paper, assuming that the measurement operator $\mc A$ has
a nested structure as described below, we address the problem of estimating
the simultaneously low-rank and row-wise sparse $\mb X^{\star}$ from
the linear measurements (\ref{eq:measurements}). Throughout the paper we may refer row-wise sparse matrices simply as sparse matrices unless explicitly stated otherwise.
 The nested operator
$\mc A$ is the product of a linear operator and a matrix which
we assume to be \emph{restricted (near) isometries} for low-rank and
row-wise sparse matrices, respectively. The inner matrix component in $\mc{A}$ compresses the columns of the target matrix without destroying its low-rank structure. This compression can be inverted because the target has sparse columns. The outer linear operator in $\mc{A}$ exploits the low-rank structure in the compressed matrix to compress the dimensions even more. We propose a computationally efficient method that relies on the nested structure of the measurement operator. We show that the proposed method is accurate uniformly for all simultaneously low-rank and row-wise sparse targets. Therefore,
the obtained accuracy bound can be interpreted as an upper bound for
the minimax rate as well. Furthermore, using standard information
theoretic tools we establish a minimax lower bound that holds for
any well-bounded linear operator $\mc A$, nested or otherwise. As
will be seen below, the derived minimax bounds differ only by a polylogarithmic
factor of the dimensions which confirm (near) optimality of the proposed
estimation method.

\subsection{Notation}

Let us first set the notation used throughout the paper. Matrices
and vectors are denoted by bold capital and small letters, respectively.
The Hermitian adjoint of linear operators and matrices is denoted
by a superscript asterisk such as $\mc M^{*}$ and $\mb M^{*}$. The
set of positive integers less than or equal to $n$ is denoted by
$\left[n\right]$. For any $a\times b$ matrix $\mb M$ and a set
$S\subseteq\left[b\right]$, the $a\times\left|S\right|$ restriction
of $\mb M$ to the columns indexed by $S$ is denoted by $\mb M_{S}$.
In particular, $\mb I_{a,S}$ denotes the restriction of the $a\times a$
identity matrix to the columns indexed by $S$. The all-zero and all-one
matrices of size $a\times b$ are denoted by $\mb 0_{a\times b}$
and $\mb 1_{a\times b}$, respectively. The Kronecker product of matrices
$\mb M_{1}$ and $\mb M_{2}$ is denoted by $\mb M_{1}\otimes\mb M_{2}$.
The orthogonal complement of any subspace $V$ is denoted by $V^{\perp}$.
The maximum and minimum of two numbers $a$ and $b$ are denoted by
$a\vee b$ and $a\wedge b$, respectively. The notation $f=\msf O\left(g\right)$
is used when $f=cg$ for some absolute constant $c>0$. We may also
use $f\lesssim g$ (or $f\gtrsim g$) to denote $f\leq cg$ (or $f\geq cg$)
for some absolute constant $c>0$. The largest integer that is less
than or equal to a real number $a$ is denoted by $\left\lfloor a\right\rfloor $.
Similarly, $\left\lceil a\right\rceil $ denotes the smallest integer
greater than or equal to $a$. The function $d_{\msf H}\left(\cdot,\cdot\right)$
denotes the natural Hamming distance between a pair of objects that
can be represented uniquely by finite binary sequences of the same
length (e.g., binary vectors, binary matrices, subsets of a finite
set, etc). For any matrix $\mb M$, the number of nonzero rows, the
Frobenius norm, the nuclear norm, and the sum of the row-wise $\ell_{2}$-norms
(i.e., the $\ell_{1,2}$-norm) are denoted by $\left\Vert \mb M\right\Vert _{0,2}$,
$\left\Vert \mb M\right\Vert _{F}$, $\left\Vert \mb M\right\Vert _{*}$,
and $\left\Vert \mb M\right\Vert _{1,2}$, respectively.

\subsection{\label{sub:Problem-setup}Problem setup}

We would like to estimate a $p_{1}\times p_{2}$ matrix $\mb X^{\star}$
from noisy linear measurements given by (\ref{eq:measurements}) where
the measurement operator $\mc A:\mbb R^{p_{1}\times p_{2}}\to\mbb R^{n}$
is assumed to have the nested form 
\begin{align}
\mc A\left(\mb X\right) & =\mc W\left(\mb{\varPsi}\mb X\right),\label{eq:nested-measurements}
\end{align}
 with the matrix $\mb{\varPsi}\in\mbb R^{m\times p_{1}}$ and the
linear operator $\mc W:\mbb R^{m\times p_{2}}\to\mbb R^{n}$ known.
We also assume throughout the paper that $\mb z$ in (\ref{eq:measurements})
is a Gaussian noise with $\msf N\left(\mb 0_{n\times1},\sigma^{2}\mb I\right)$
distribution. The estimation problem is assumed to be in the high-dimensional
regime (i.e., $n\ll p_{1}p_{2}$). We consider the target matrices
to have rank at most $r\ll p_{1}\wedge p_{2}$ and at most $k\ll p_{1}$
nonzero rows, where clearly $r\leq k$. The set of all such matrices
can be explicitly defined as 
\begin{align*}
\mbb X{}_{k,r} & \coloneqq\left\{ \mb X\in\mbb R^{p_{1}\times p_{2}}\mid\left\Vert \mb X\right\Vert _{0,2}\leq k,\ \msf{rank}\left(\mb X\right)\leq r\right\} ,
\end{align*}
 where $\left\Vert \mb X\right\Vert _{0,2}$ denotes the number of
nonzero rows of $\mb X$. We denote the set of matrices with rank
at most $r$ by $\mbb X_{\tbullet,r}$ and the set of matrices with
at most $k$ nonzero rows by $\mbb X_{k,\tbullet}$. Furthermore,
we assume that the rank parameter $r=\msf{rank}\left(\mb X^{\star}\right)$
and the noise variance $\sigma^{2}$ are known to the estimator.

The accuracy guarantees of our method can be treated as a minimax
upper bound, if they hold uniformly for all target matrices in $\mbb X_{k,r}$.
To derive the desired uniform accuracy guarantees we assume that $\mb{\varPsi}$
is a restricted isometry for sparse matrices and $\mc W$ is
a restricted isometry for low-rank matrices which are common assumptions
for compressive sensing and low-rank matrix estimation \cite{Baraniuk_Simple_2008,Recht_Guaranteed_2010,Candes_Tight_2011}.
 To be precise, for any integer $k\in\left[p_{1}\right]$, $\mb{\varPsi}$
is a said to be a restricted isometry over $\mbb X_{k,\tbullet}$
if there exists a \emph{restricted isometry constant} $\delta_{k,\tbullet}\left(\mb{\varPsi}\right)\in\left[0,1\right]$
such that 
\begin{align*}
\forall\mb X\in\mbb X_{k,\tbullet}:\ \left(1-\delta_{k,\tbullet}\left(\mb{\varPsi}\right)\right)\left\Vert \mb X\right\Vert _{F}^{2}\leq\left\Vert \mb{\varPsi}\mb X\right\Vert _{F}^{2} & \leq\left(1+\delta_{k,\tbullet}\left(\mb{\varPsi}\right)\right)\left\Vert \mb X\right\Vert _{F}^{2}.
\end{align*}
 Clearly, any restricted isometry for $k$-sparse vectors is a restricted isometry over $\mbb{X}_{k,\tbullet}$ and vice versa. Similarly, for all integers $r\in\left[m\wedge p_{2}\right]$, $\mc W$
is said to be the restricted isometry over $\mbb X_{\tbullet,r}$
if there exists a restricted isometry constant $\delta_{\tbullet,r}\left(\mc W\right)\in\left[0,1\right]$
such that 
\begin{align*}
\forall\mb X\in\mbb X_{\tbullet,r}:\ \left(1-\delta_{\tbullet,r}\left(\mc W\right)\right)\left\Vert \mb X\right\Vert _{F}^{2}\leq\left\Vert \mc W\left(\mb X\right)\right\Vert _{2}^{2} & \leq\left(1+\delta_{\tbullet,r}\left(\mc W\right)\right)\left\Vert \mb X\right\Vert _{F}^{2}.
\end{align*}

The goal of our minimax approach is to characterize the tolerance
level $\tau$ in terms of the parameters of the estimation problem,
such that the ``minimax failure probability'' given by 
\begin{align*}
\inf_{\widehat{\mb X}}\sup_{\mb X^{\star}\in\mbb X_{k,r}} & \P\left(\left\Vert \widehat{\mb X}-\mb X^{\star}\right\Vert _{F}\geq\tau\right)
\end{align*}
is sufficiently small. In words, $\tau$ would be the worst-case accuracy
of the best estimator that holds with high probability. As will be
seen in Section \ref{sec:Main-Results}, we use the restricted isometry
properties mentioned above to provide a uniform accuracy guarantee
for our proposed method which also serves as a minimax upper bound.
To establish the minimax lower bound, we employ some of the standard
information theoretic tools that help to determine non-trivial values
of $\tau$ for which the above minimax failure probability is significant,
e.g., more than $\frac{1}{2}$.  Derivation of the minimax lower bound
does not depend on the restricted isometry assumptions and  we merely
 assume that 
\begin{align}
\forall\mb X\in\mbb X_{k,r}:\ \left\Vert \mc A\left(\mb X\right)\right\Vert _{2}^{2} & \leq\gamma_{k,r}\left\Vert \mb X\right\Vert _{F}^{2},\label{eq:A-bounded}
\end{align}
 for some constant $\gamma_{k,r}>0$. Of course, if $\mc A$ is defined
by (\ref{eq:nested-measurements}) and we happen to have the restricted
isometries $\mb{\varPsi}$ and $\mc W$ as above, clearly there exists a constant $\gamma_{k,r}$ such that $\gamma_{k,r}\leq\left(1+\delta_{\tbullet,r}\left(\mc W\right)\right)\left(1+\delta_{k,\tbullet}\left(\mb{\varPsi}\right)\right)$.

\subsection{\label{sub:Contribution}Contributions}

\begin{itemize}
\item \textbf{Minimax upper bound:}

Under the observation model described by (\ref{eq:measurements})
and (\ref{eq:nested-measurements}) we propose a method that
\begin{quote}
\textit{With $n=\msf O\left(r\left(m\vee p_{2}\right)\right)$ and
under the assumption that $\mc W$ and $\mb{\varPsi}$ are restricted
isometries over rank-$4r$ matrices (i.e., $\mbb X_{\tbullet,4r}$)
and row-wise $2k$-sparse matrices (i.e., $\mbb X_{2k,\tbullet}$),
respectively, produces some estimate $\widehat{\mb X}$ of $\mb X^{\star}$
such that 
\begin{align*}
\left\Vert \widehat{\mb X}-\mb X^{\star}\right\Vert _{F} & \lesssim\sigma\sqrt{r\left(m\vee p_{2}\right)},
\end{align*}
holds uniformly for all $\mb X^{\star}\in\mbb X_{k,r}$.}
\end{quote}

In particular, we show that if $\mb{\varPsi}$ and $\mc W$ are Gaussian
with properly scaled iid entries then 
\begin{quote}
\textit{With $m=\msf O\left(k\,\log\frac{p_{1}}{k}\right)$ and $n=\msf O\left(r\left(m\vee p_{2}\right)\right)$
the desired restricted isometry properties hold and the produced estimate
obeys 
\begin{align*}
\left\Vert \widehat{\mb X}-\mb X^{\star}\right\Vert _{F} & \lesssim\sigma\sqrt{r\left(k\log\frac{p_{1}}{k}\vee p_{2}\right)},
\end{align*}
for all $\mb X^{\star}\in\mbb X_{k,r}$.}
\end{quote}

Because the above accuracy results hold uniformly for all targets
in $\mbb X_{k,r}$, they can also be viewed as a minimax upper bound.

\item \textbf{Minimax lower bound:}

We also establish a minimax lower bound merely under the
assumption that the linear measurement operator $\mc A$ satisfies
(\ref{eq:A-bounded}) whether it takes the nested form (\ref{eq:nested-measurements})
or not. Namely, we show that 
\begin{quote}
\textit{If $\mc A$ obeys (\ref{eq:A-bounded}) then for a sufficiently
small absolute constant $c>0$ and for any estimator $\widehat{\mb X}$
there exists a target $\mb X^{\star}\in\mbb X_{k,r}$ such that 
\begin{align*}
\left\Vert \widehat{\mb X}-\mb X^{\star}\right\Vert _{F} & \geq c\sigma\sqrt{\frac{k\log\frac{p_{1}}{k}+r\left(k\vee p_{2}\right)}{\gamma_{k,r}}}
\end{align*}
 holds with probability at least $\frac{1}{2}$.}
\end{quote}
\end{itemize}
With $\gamma_{k,r}=\msf O\left(1\right)$, the obtained minimax lower
and upper bounds differ only by some logarithmic factors, which shows
that by exploiting the nested form of the measurements the proposed
estimator achieves a near optimal minimax rate. For comparison, if
we only considered the low-rank structure, then for any estimator
the estimation error in Frobenius norm would have been bounded by
$\msf O\left(\sigma\sqrt{r\left(p_{1}\vee p_{2}\right)}\right)$ from
below (cf. \cite{Rohde_Estimation_2011} and \cite[Theorem 2.5]{Candes_Tight_2011}).
Similarly, considering only the sparsity as the structure of $\mb X^{\star}$,
the accuracy lower bounds for sparse estimation (e.g.,\cite[Theorem 1]{Raskutti_Minimax_2011})
suggest the lower bound $\msf O\left(\sigma\sqrt{p_{2}k\,\log\frac{p_{1}}{k}}\right)$
for the estimation error. Therefore, with the natural assumption that
$p_{1}\gg p_{2}$, these lower bounds indicate that exploiting the
low-rank and sparse structures simultaneously has reduced the estimation
error dramatically. Furthermore, in the noise-free scenario, our results
suggest that we can recover $\mb X^{\star}$ exactly from $\msf O\left(r\left(k\log\frac{p_{1}}{k}\vee p_{2}\right)\right)$
measurements. Therefore, the nested structure of the measurements
and the proposed estimation procedure have a critical role in achieving
the mentioned sample complexity which cannot be achieved by minimization
of any convex combination of the nuclear norm and the $\ell_{1,2}$-norm
for many common types of linear measurements \cite{Oymak_Simultaneously_2015}.

\subsection{Related work}

Many important problems in statistics and machine learning such as
Sparse Principal Component Analysis (SPCA) can be formulated as estimation
of simultaneously low-rank and sparse matrices. These estimation problems
are significantly more challenging than estimation of matrices that
are either low-rank or sparse. In \emph{sparse spiked covariance}
estimation \cite{Johnstone_Consistency_2009} and more generally SPCA,
given a relatively small number of independent samples of a mean-zero
multivariate Gaussian, the primary goal is to estimate the principal
eigenvectors of the associated covariance matrix that are assumed
to be sparse. Many convex and non-convex algorithms have been proposed
for these problems including but not limited to \cite{d'Aspremont_Direct_2007,Johnstone_Consistency_2009,Amini_High-dimensional_2009,Cai_Sparse_2013,Vu_Fantope_2013,Wang_Tighten_2014,Cai_Optimal_2015}.
Each of these algorithms provides a different trade-off between statistical
optimality and computational complexity. However, as shown in \cite{Berthet_Complexity_2013},
there may be an inevitable gap between the optimal statistical accuracy
and the statistical accuracy that polynomial-time algorithms can achieve
in detection of sparse principal components. Since SPCA can be viewed
as a primitive for solving linear regression with simultaneously low-rank
and sparse target matrices, the result of \cite{Berthet_Complexity_2013}
can also indicate the difficulty of the latter problem.

The problem of denoising low-rank and row-wise sparse matrices is
also considered in \cite{Buja_Optimal_2013}. The iterative algorithm
proposed in \cite{Buja_Optimal_2013} is basically an adaptation of
the \emph{power iteration} where the factors are sparsified through
thresholding. It is shown, that under mild conditions, this iterative
method can be initialized appropriately and then is shown to be nearly
minimax optimal. A more related work to our problem is multiple linear
regression considered in \cite{Ma_Adaptive_2014} where the goal is
to estimate $\mb X^{\star}$ that is low-rank and row-wise sparse
from the measurements $\mb Y=\mb A\mb X^{\star}+\mb Z$ with $\mb A$
and $\mb Z$ being the design and the noise matrices, respectively.
The algorithm of \cite{Ma_Adaptive_2014} uses the low-rank structure
of $\mb X^{\star}$ to construct an initial estimate $\mb V$ of the
right singular vectors of $\mb X^{\star}$. Using this estimate, the
algorithm then solves a least squares with a convex regularization
for row-wise sparsity to estimate $\mb B$, the projection of $\mb X^{\star}$
onto the estimated singular vectors. Then $\mb V$ is updated to be
the right singular vectors of $\mb Y$ after projecting its range
onto the range of $\mb B$. Finally, the algorithm updates $\mb B$
by repeating the mentioned convex optimization and outputs $\mb B\mb V^{\msf T}$
as the estimate of $\mb X^{\star}$ which is shown to be nearly minimax
optimal. 

Furthermore, an efficient alternating minimization method for estimation
of rank-one and sparse matrices from linear measurements is proposed
in \cite{Lee_Near_2013}. The algorithm is shown to be accurate and
nearly optimal in terms of sample complexity, provided that the factors
of the target rank-one matrix have relatively dominant spikes and
the noise level is moderate. Estimation of simultaneously structured
matrices from compressive linear measurements is also studied in \cite{Oymak_Simultaneously_2015}.
It is shown in \cite{Oymak_Simultaneously_2015} that using convex
proxies for each of the assumed structures independently would result
in a suboptimal sample complexity. Specifically, it is shown \cite{Oymak_Simultaneously_2015}
that minimization of any mixture of the $\ell_{1,2}$-norm and the
nuclear norm fails to recover the true signal if there are less than
$\msf O\left(kp_{2}\wedge r\left(p_{1}+p_{2}\right)\right)$ measurements.
Therefore, our results as stated above show a significant improvement
over the na\"{\i}ve convex relaxation method studied in \cite{Oymak_Simultaneously_2015}.

\section{\label{sec:Main-Results}Main Results}
\subsection{Two-stage estimator}
We propose the following two-stage method for estimation of simultaneously
low-rank and row-wise sparse matrices from measurements given by (\ref{eq:measurements})
when $\mc A$ has the nested structure described by (\ref{eq:nested-measurements}).
The nuclear norm and the sum of row-wise $\ell_{2}$-norm are denoted
by $\left\Vert \cdot\right\Vert _{*}$ and $\left\Vert \cdot\right\Vert _{1,2,}$,
respectively.
\begin{enumerate}
\item Low-rank estimation stage:
\begin{align}
\widehat{\mb B} & \in\argmin_{\mb B}\left\Vert \mb B\right\Vert _{*}\label{eq:pre-estimate}\\
 & \text{subject to }\left\Vert \mc W\left(\mb B\right)-\mb y\right\Vert _{2}\leq\sigma\sqrt{n+c_{1}r\left(m\vee p_{2}\right)}\nonumber 
\end{align}

\item Sparse estimation stage:
\begin{align}
\widehat{\mb X} & \in\argmin_{\mb X}\left\Vert \mb X\right\Vert _{1,2}\label{eq:estimate}\\
 & \text{subject to }\left\Vert \mb{\varPsi}\mb X-\widehat{\mb B}\right\Vert _{F}\leq c_{2}\sigma\sqrt{r\left(m\vee p_{2}\right)}\nonumber 
\end{align}

\end{enumerate}
With appropriately chosen constants $c_{1}$ and $c_{2}$ each stage
is a convex program for which many efficient solvers exist. 
\paragraph{Variations of the estimator:}
The specific form of the optimization that expresses the
low-rank estimation and the sparse estimation stages is not critically
important. It is only necessary to produce a solution that is sufficiently
accurate in each stage. For example, with the appropriate regularization
parameter $\lambda>0$, we could have used the regularized analog
of (\ref{eq:pre-estimate}) given by 
\begin{align*}
\widehat{\mb B} & \in\argmin_{\mb B}\frac{1}{2}\left\Vert \mc W\left(\mb B\right)-\mb y\right\Vert _{2}^{2}+\lambda\left\Vert \mb B\right\Vert _{*}
\end{align*}
 which also enjoys the desired accuracy guarantees \cite{Candes_Tight_2011}.
Similarly, the sparse estimation stage can be performed by the regularized
form of (\ref{eq:estimate}). Furthermore, because we are assuming
that $\mb{\varPsi}$ and $\mc W$ are restricted isometries, the convex
programs considered in (\ref{eq:pre-estimate}) and (\ref{eq:estimate})
could be replaced by non-convex greedy algorithms such as the \emph{iterative
hard thresholding} (IHT) \cite{Blumensath_Iterative_2009} that achieve the same accuracy usually at a lower
computational cost. These methods, however, usually require tighter
bounds on the restricted isometry constants.

\paragraph{Post-processing:}
The solutions obtained by (\ref{eq:pre-estimate}) and (\ref{eq:estimate})
generally are not low-rank or sparse. To enforce the desired structures,
after each stage of the estimator we can simply project the estimator
onto the set of low-rank and/or row-wise sparse matrices. It is straightforward
to show that these post-processing steps will not change the derived
error bounds beyond a constant factor. Furthermore, we can treat range
of the best rank-$r$ approximation to $\widehat{\mb B}$ as an estimate
for the range of $\mb{\varPsi}\mb X^{\star}$ and pass it to a sparse
estimation stage similar to (\ref{eq:estimate}), but with an optimization
variable that has $r$ columns rather than $p_{2}$. Therefore, we
can significantly reduce the computational cost of the second stage
of the estimator. However, to analyze the performance of this modified
estimator it is necessary to convert the range estimation to the actual
estimation error which we do not pursue in this paper.
\subsection{Accuracy of the estimator and the minimax upper bound}
In this section we state our result on the statistical accuracy of
the two-stage estimator described by (\ref{eq:pre-estimate}) and
(\ref{eq:estimate}). This accuracy guarantee can be viewed as a minimax
upper bound as well. 

For $\mb z\sim\msf N\left(\mb 0_{n\times1},\sigma^{2}\mb I\right)$,
tail bounds for chi-squared random variables \cite[Lemma 1]{Laurent_Adaptive_2000}
guarantee that for any $\nu>0$ we have 
\begin{align*}
\sigma^{2}\left(n-\nu\right)\leq\left\Vert \mb z\right\Vert _{2}^{2} & \leq\sigma^{2}\left(n+\nu\right)
\end{align*}
with probability at least $1-2\exp\left(-\left(\frac{\nu}{4}\wedge\frac{\nu^{2}}{16n}\right)\right)$.
Therefore, assuming that $n=\msf O\left(r\left(m\vee p_{2}\right)\right)$
and choosing $\nu=c_{1}r\left(m\vee p_{2}\right)$ for a sufficiently
large absolute constant $c_{1}>0$ we have 
\begin{align}
\sigma^{2}\left(n-c_{1}r\left(m\vee p_{2}\right)\right)\leq\left\Vert \mb z\right\Vert _{2}^{2} & \leq\sigma^{2}\left(n+c_{1}r\left(m\vee p_{2}\right)\right),\label{eq:noise-bound}
\end{align}
 with high probability. Therefore, we can guarantee that the matrix
$\mb B^{\star}=\mb{\varPsi}\mb X^{\star}$ is in the feasible set
of (\ref{eq:pre-estimate}) with high probability. Consequently, we
can invoke Lemma \ref{lem:pre-estimate} and guarantee that $\mb X^{\star}$
is in the feasible set of (\ref{eq:estimate}) for appropriate choice
of the constant $c_{2}>0$. Of course, the constants $c_{1}$ and
$c_{2}$ depend on the restricted isometry constants of $\mc W$ and
$\mb{\varPsi}$.
\begin{thm}[minimax upper bound]
\label{thm:upper-bound} Let $n=\msf O\left(r\left(m\vee p_{2}\right)\right)$.
Furthermore, suppose that $\mb{\varPsi}$ and $\mc W$ have sufficiently
small restricted isometry constants $\delta_{\tbullet,2k}\left(\mb{\varPsi}\right)$
and $\delta_{4r,\tbullet}\left(\mc W\right)$, respectively. Then,
for appropriately chosen constants $c_{1}$ and $c_{2}$, there exists
an absolute constant $C>0$ depending on $c_{1}$ and $c_{2}$ such
that the estimate produced using (\ref{eq:pre-estimate}) and (\ref{eq:estimate})
obeys 
\begin{align*}
\left\Vert \widehat{\mb X}-\mb X^{\star}\right\Vert _{F} & \leq C\sigma\sqrt{r\left(m\vee p_{2}\right)}
\end{align*}
for all $\mb X^{\star}\in\mbb X_{k,r}$ with high probability.
\end{thm}
The proof of Theorem \ref{thm:upper-bound} is straightforward and
provided in the appendix. If $\mb{\varPsi}$ and $\mc W$ are drawn
from certain ensembles of random matrices or operators (e.g., Gaussian,
Rademacher, partial Fourier, partial random circulant, etc.), then
with 
\begin{align*}
m & =\msf O\left(k\,\mr{polylog}\left(p_{1},k\right)\right) & \text{and} &  & n & =\msf O\left(r\left(k\vee p_{2}\right)\mr{polylog}\left(p_{1},p_{2},k,r\right)\right),
\end{align*}
where $\mr{polylog\left(\right)}$ denotes a polylogarithmic factor
of its argument, we can guarantee the desired restricted isometry
properties with high probability. Therefore, an immediate consequence
of Theorem \ref{thm:upper-bound} in these scenarios is that we can
guarantee estimation error of the order $\sqrt{r\left(k\vee p_{2}\right)\mr{polylog}\left(p_{1},p_{2},k,r\right)}$.
To have a concrete example, the case of Gaussian operators is addressed
by the following corollary.
\begin{cor}
\label{cor:Gaussian}Suppose that $\mb{\varPsi}$ has iid $\msf N\left(0,\frac{1}{m}\right)$
entries. Furthermore, suppose that $\mc W$ is a Gaussian operator
that simply takes the inner product of its argument with $n$ independent
Gaussian matrices each populated with iid $\msf N\left(0,\frac{1}{n}\right)$
entries. If $m=C_{1}k\log\frac{p_{1}}{k}$ and $n=C_{2}r\left(m\vee p_{2}\right)$
for sufficiently large absolute constants $C_{1}$ and $C_{2}$, then
there exists an absolute constant $C>0$ such that the estimate $\widehat{\mb X}$
obtained using (\ref{eq:pre-estimate}) and (\ref{eq:estimate}) satisfies
\begin{align*}
\left\Vert \widehat{\mb X}-\mb X^{\star}\right\Vert _{F} & \leq C\sigma\sqrt{r\left(k\log\frac{p_{1}}{k}\vee p_{2}\right)},
\end{align*}
for all $\mb X^{\star}\in\mbb X_{k,r}$ with high probability.\end{cor}
\begin{proof}
With $m=\msf O\left(k\log\frac{p_{1}}{k}\right)$ it can be shown
through the standard covering argument and the union bound that, with
high probability, the considered $\mb{\varPsi}$ would be a restricted
isometry over $2k$-sparse vectors in $\mbb R^{p_{1}}$ \cite{Baraniuk_Simple_2008}.
Clearly, in this case $\mb{\varPsi}$ is also a restricted isometry
over $\mbb X_{2k,\tbullet}$ as it preserves the $\ell_{2}$-norm
of the columns of any matrix in $\mbb X_{2k,\tbullet}$. Similarly,
$\mc W$ would be a restricted isometry over rank-$r$ matrices in
$\mbb R^{m\times p_{2}}$ if $n=\msf O\left(r\left(m\vee p_{2}\right)\right)$
\cite[Theorem 2.3]{Candes_Tight_2011}. Then, the desired accuracy
bound follows immediately from Theorem \ref{thm:upper-bound}.
\end{proof}

\paragraph{Weaker assumptions and non-uniform guarantees:}
The assumptions that the linear operator $\mc W$ and the
matrix $\mb{\varPsi}$ are restricted isometries are primarily used
to establish the minimax bounds that are valid uniformly for all target
matrices $\mb X^{\star}\in\mbb X_{k,r}$. However, these assumptions
are not generally needed if the goal is merely to show accuracy of
the proposed two-stage method for any particular instance of the problem.
For example, for certain random operators $\mc W$ and matrices $\mb{\varPsi}$
that cannot be restricted isometries, the accuracy of (\ref{eq:measurements})
and (\ref{eq:pre-estimate}) can be shown by construction of a \emph{dual
certificate} through the \emph{golfing scheme} \cite{Gross_Recovering_2011}.
In these regimes, however, the robustness to noise is often weaker.

\paragraph{Low-rank and column-sparse matrices:}
 If the target low-rank matrices are column-sparse rather than row-wise sparse our results still apply with minor adjustments. if the columns of a low-rank matrix are $k$-sparse, but not necessarily supported on the same rows, the restricted isometry of $\mb{\varPsi}$ still holds. Therefore, we can follow a similar algorithm in which the $\ell_{1,2}$ norm in the second stage is replaced by another norm such as the maximum column-wise $\ell_1$  norm. The lower bound still holds as well since row-wise sparse matrices are a special instances of column-sparse matrices.

\subsection{The minimax lower bound}
The following theorem provides a minimax lower bound for the probability
of estimation failure over $\mbb X_{k,r}$. 
\begin{thm}[minimax lower bound]
\label{thm:lower-bound} Suppose that $\mc A$ obeys (\ref{eq:A-bounded}).
Let $\widehat{\mb X}$ denote any estimator of $\mb X^{\star}$ based
on the measurements of the form (\ref{eq:measurements}). Then, there
exists a sufficiently small absolute constant $c>0$ such that with
a  probability more than one half the estimation error over $\mbb X_{k,r}$
exceeds $c\sigma\sqrt{\frac{k\log\frac{p_{1}}{k}+r\left(k\vee p_{2}\right)}{\gamma_{k,r}}}$.
Namely, we have the minimax lower bound 
\begin{align*}
\inf_{\widehat{\mb X}}\sup_{\mb X^{\star}\in\mbb X_{k,r}}\P\left(\left\Vert \widehat{\mb X}-\mb X^{\star}\right\Vert _{F}\geq c\sigma\sqrt{\frac{k\log\frac{p_{1}}{k}+r\left(k\vee p_{2}\right)}{\gamma_{k,r}}}\right) & >\frac{1}{2}.
\end{align*}
\end{thm}
In words, Theorem \ref{thm:lower-bound} shows that if $\gamma_{k,r}=\msf O\left(1\right)$,
then the error of any estimator of the matrices in $\mbb X_{k,r}$
cannot be uniformly better than $\msf O\left(\sigma\sqrt{k\log\frac{p_{1}}{k}+r\left(k\vee p_{2}\right)}\right)$
with high probability. For instance, this result implies that the
bound established in Corollary \ref{cor:Gaussian} is near-optimal
as the error bound is within $\log\left(\frac{p_{1}}{k}\right)$ factor
of the lower bound. We would like to emphasize that, we did not assume
the nested structure (\ref{eq:nested-measurements}) for $\mc A$
to prove the lower bound. Therefore, the lower bound  applies to any
linear operator $\mc A$ obeying (\ref{eq:A-bounded}) regardless
of whether it is of the form (\ref{eq:nested-measurements}) or not.

\section{Applications and Extensions}
\subsection{Blind deconvolution with randomly coded masks}
The linear measurements of the form (\ref{eq:nested-measurements})
appear in some interesting and important applications in computational
imaging. For example, (\ref{eq:nested-measurements}) can describe
the measurements in the lifted formulation of blind deconvolution
problems that arise in imaging systems with randomly coded masks \cite{Tang_Convex_2014,Bahmani_Lifting_2014}.
In these problems rank-one target matrix is $\mb X=\mb u\mb h^{*}$
with $\mb u$ and $\mb h$ being the target image and the Discrete
Fourier Transform (DFT) of the unknown blurring kernel, respectively.
The blind deconvolution can be solved effectively, by estimating the
matrix $\mb X$ from the linear measurements of the form 
\begin{align*}
\mc A\left(\mb X\right) & =\mb{\varPhi}^{*}\left(\mb X\circ\mb F\right).
\end{align*}
where $\mb F$ denotes the DFT matrix, $\circ$ denotes the Hadamard
product, and $\mb{\varPhi}$ is a matrix whose columns model the random
masks that modulate the image $\mb u$. We assume that the image $\mb u$
is sparse with respect to an orthonormal basis that is \emph{incoherent
}with the canonical basis. This assumption is realistic because, if
necessary, spectral modulation can always create an effective sparsifying
basis for the target image that has the desired incoherence condition.
To simplify the exposition, let us assume that the image is sparse
in the DFT basis and we have $\mb u=\mb F\tilde{\mb u}$ with $\widetilde{\mb u}$
being a $k$-sparse vector. Therefore, instead of $\mb X$ one may
consider $\widetilde{\mb X}=\widetilde{\mb u}\mb h^{*}$ as the target
rank-one matrix which is also row-wise sparse. Now if the random masks
$\mb{\varPhi}$ are designed to suppress all but a small number of
randomly chosen pixels indexed by $\varOmega$, then the measurement
model reduces to 
\begin{align*}
\widetilde{\mc A}\left(\widetilde{\mb X}\right) & =\mb{\varPhi}_{\Omega}^{*}\left(\mb F_{\Omega}\widetilde{\mb X}\circ\mb F_{\varOmega}\right),
\end{align*}
where the subscript $\varOmega$ denotes restriction to the rows indexed
by $\varOmega$. It is now clear that with $\mb{\varPsi}=\mb F_{\varOmega}$
and $\mc W:\mb B\mapsto\msf{vec}\left(\mb{\varPhi}_{\Omega}^{*}\left(\mb B\circ\mb F_{\varOmega}\right)\right)$,
the above equation is a special case of (\ref{eq:nested-measurements})
up to a trivial vectorization.

\subsection{Low-rank and doubly-sparse matrices}
In this paper, we chose to state the main theoretical results
only for the low-rank and row-wise sparse model mentioned in Section
\ref{sub:Problem-setup}. However, these results can also be easily
extended to the case of low-rank and doubly sparse matrices, where
the target matrix is sparse both row-wise and column-wise. These kinds of matrices occur, for instance, in compressive phase retrieval, elaborated on in the following subsection, and \emph{covariance sketching} \cite{Bahmani_Sketching_2015}. The appropriate
nested form of the measurement operator for these problems is 
\begin{align}
\mc A\left(\mb X\right) & =\mc W\left(\mb{\varPsi}_{1}\mb X\mb{\varPsi}_{2}^{*}\right),\label{eq:doubly-sparse}
\end{align}
 with $\mc W$, $\mb{\varPsi}_{1}$, and $\mb{\varPsi}_{2}$ given.  
The only modification needed for the estimator would be to use $\mb{\varPsi}_{1}$
and perform the first sparse estimation stage for row-wise sparsity
as before, and then use $\mb{\varPsi}_{2}$ to perform a second sparse
estimation stage for column-wise sparsity.

\subsection{Compressive phase retrieval}\label{ssec:CPR}
An interesting special case of the extension described above
by (\ref{eq:doubly-sparse}) is when $\mb{\varPsi}_{1}=\mb{\varPsi}_{2}$
and the linear operator $\mc W$ measures inner product of its argument
and some rank-one matrices $\mb w_{i}\mb w_{i}^{*}$ (i.e., $\mc W:\mb B\mapsto\left[\left\langle \mb w_{i}\mb w_{i}^{*},\mb B\right\rangle \right]_{i=1}^{n}$).
In particular, this model can be used in Compressive Phase Retrieval
(CPR) \cite{Moravec_Compressive_2007,Shechtman_Sparsity_2011} where
the target matrix $\mb X^{\star}=\mb x\mb x^{*}$ with a sparse $\mb x\in\mbb C^{p}$
is estimated from measurements $\left[\left|\left\langle \mb a_{i},\mb x\right\rangle \right|^{2}\right]_{i=1}^{n}$.
Therefore, the measurement operator in CPR can be written as $\mc A:\mb X\mapsto\left[\left\langle \mb a_{i}\mb a_{i}^{*},\mb X\right\rangle \right]_{i=1}^{n}$.
The CPR problem is posed as non-convex optimization problems in \cite{Moravec_Compressive_2007,Shechtman_Sparsity_2011,Schechtman_GESPAR_2014}
where the sparsity of the solution is minimized subject to a constraint
on the quartic prediction error or vice versa. While certain local
convergence guarantees are established for the GESPAR algorithm \cite{Schechtman_GESPAR_2014},
global convergence and statistical accuracy of the proposed algorithm
remains unknown. In \cite{Netrapalli_Phase_2013}, another non-convex
approach based on alternating minimization is proposed for the standard
Phase Retrieval (PR) as well as the CPR. With the particular initialization
proposed in \cite{Netrapalli_Phase_2013}, the alternating minimization
method is shown to converge linearly in the noise-free PR and CPR
problems. However, this convergence rate holds for CPR when the number
of measurements grows quadratically in $k$. For large-scale problems
this alternating minimization method would still be a favorite choice
among the competing algorithms as it is computationally less demanding.
Furthermore, assuming iid Gaussian measurement vectors $\mb a_{i}$,
\cite{Ohlsson_CPRL_2012} and \cite{Li_Sparse_2013} consider a convex
relaxation to the lifted CPR problem formulated as 
\begin{align*}
\begin{aligned} & \argmin_{\mb X\succcurlyeq\mb 0} &  & \msf{trace}\left(\mb X\right)+\lambda\left\Vert \mb X\right\Vert _{1}\\
 & \text{subject to} &  & \mc A\left(\mb X\right)=\mb y,
\end{aligned}
\end{align*}
for some parameter $\lambda\geq0$. For the considered type of measurements,
\cite{Li_Sparse_2013} shows that the above convex program can recover
the target sparse and rank-one matrix $\mb X^{\star}$ if the number
of measurements depend quadratically on $k$, the number of the nonzeros
of the signal $\mb x$. The same result is shown in \cite{Chen_Exact_2015}
for sub-Gaussian measurements and through a simpler derivation by
``debiasing'' the measurements and showing the obtained operator
obeys an $\ell_{1}$/$\ell_{2}$ variant of the restricted isometry
property. Furthermore, \cite{Li_Sparse_2013} demonstrates that the
quadratic dependence on $k$ cannot be improved fundamentally by varying
the coefficient $\lambda$ which is in agreement with the results
of \cite{Oymak_Simultaneously_2015}.

 As mentioned above, if we choose $\mb a_{i}=\mb{\varPsi}^{*}\mb w_{i}$ for some compressive
sensing matrix $\mb{\varPsi}$, and random vectors $\left\{ \mb w_{i}\right\} _{i=1}^{n}$,
then the measurement operator takes the nested form $\mc A:\mb X\mapsto\left[\left\langle \mb w_{i}\mb w_{i}^{*},\mb{\varPsi}\mb X\mb{\varPsi}^{*}\right\rangle \right]_{i=1}^{n}$. It is then straightforward to apply the framework developed in this
paper and show that $\mb X^{\star}$ can be reconstructed accurately
and efficiently from $n\ll p$ measurements obtained by $\mc A$  through the two-stage recovery
\begin{align}
\widehat{\mb B} & \in\argmin_{\mb B\succcurlyeq\mb{0}} \msf{trace}\left(\mb B\right)\label{eq:cpr-lr}\\
 & \text{subject to }\sum_{i=1}^n \left(\mb{w}_i^* \mb B\mb{w}_i-y_i \right)^2\leq \varepsilon ^2\nonumber \\
\widehat{\mb X} & \in\argmin_{\mb X}\left\Vert \mb X\right\Vert _1\label{eq:cpr-sr}\\
 & \text{subject to }\left\Vert \mb{\varPsi}\mb X\mb{\varPsi}^*-\widehat{\mb B}\right\Vert _{F}\leq \frac{C\varepsilon}{\sqrt{n}}\nonumber ,
\end{align}
where $\varepsilon \ge \left\Vert \mb{z}\right\Vert_2$ is a bound on the noise, $\left\Vert\cdot\right\Vert_1$ denotes the $\ell_1$ norm, and $C$ is an absolute constant. The precise guarantees for effectiveness of nested measurements in CPR are established independently in {\cite{Iwen_Robust_2015}}, \cite{Yapar_Fast_2015}, and {\cite{Bahmani_Efficient_2015}}. In particular, the following is shown in {\cite{Bahmani_Efficient_2015}}.
\begin{thm}[{\citep[Corollary 1]{Bahmani_Efficient_2015}}]
Let $\mb{\varPsi}\in\mbb{R}^{m\times p}$ be a matrix with independent $\msf{N}(0,\frac{1}{m})$ entries, and for $i=1,2,\dotsc,n$ let $\mb{w}_i\in\mbb{R}^m$ be independent copies of a vector with i.i.d. standard Gaussian entries. For an arbitrary $k$-sparse signal $\mb{x}\in\mbb{R}^p$ we observe measurements  of the form $y_i = \langle \mb{\varPsi}^\msf{T}\mb{w}_i, \mb{x}\rangle^2 + z_i$ . If $m\ge c_1 k\log\frac{p}{k}$ and $n\ge c_2 m$ for sufficiently large absolute constant $c_1$ and $c_2$, then the two-stage recovery through \eqref{eq:cpr-lr} and \eqref{eq:cpr-sr} produces an estimate $\widehat{\mb{X}}$ that obeys\
\begin{align*}
\left\Vert\widehat{\mb{X}}-\mb{x}\mb{x}^\msf{T}\right\Vert_F &\le \frac{C'\varepsilon}{\sqrt{n}},
\end{align*}
with high probability for some absolute constant $C'>0$.
\end{thm}

The fact that the intermediate operator $\mc W:\mb B\mapsto\left[\left\langle \mb w_{i}\mb w_{i}^{*},\mb B\right\rangle \right]_{i=1}^{n}$
is generally not a restricted isometry precludes our minimax analysis.
However, as mentioned in Section \ref{sub:Contribution}, the restricted
isometry conditions are used merely to obtain uniform accuracy guarantees.
The desired uniform accuracy guarantees can be established through
other approaches, as done in \cite{Kueng_Low-rank_2014} using
the \emph{small-ball method }\cite{Mendelson_Learning_2014}. Moreover,
if we disregard uniform guarantees and thus minimax optimality, instance
accuracy guarantees of the low-rank estimation stage also can be established
through the small-ball method \cite{Tropp_Convex_2014} or the construction of a dual certificate \cite{Candes_PhaseLift_2013}. Therefore, we can show that the CPR problem can be solved in our proposed
framework accurately with a sample complexity that grows much slower than the ambient dimension of the target signal.

\section{Numerical Experiment}

We performed a set of simulations on synthetic data in which both
$\mb{\varPsi}$ and $\mc W$ are considered to be Gaussian as in Corollary
\ref{cor:Gaussian} with $m=\left\lceil 5k\log\frac{p_{1}}{k}\right\rceil $
and $n=4r\left(m\vee p_{2}\right)$. We generated target matrices
of size $1000\times30$ (i.e., $p_{1}=1000$ and $p_{2}=30$) that
are factored as $\mb X^{\star}=\mb U\mb V^{\msf T}$ where $\mb U\in\mbb R^{p_{1}\times r}$
whose $k$ nonzero rows are chosen uniformly at random and $\mb V\in\mbb R^{p_{2}\times r}$.
The nonzero entries of both $\mb U$ and $\mb V$ are independent
draws from the standard Gaussian distribution. The noise variance
is fixed at $\sigma^{2}=10^{-4}$. The rank of the target matrix is
selected from the range $1\leq r\leq10$. Similarly, the row-wise
sparsity is selected from the range $10\leq k\leq19$. For each pair
of $\left(r,k\right)$ the algorithm is tested for $100$ trials.
We used the TFOCS package \cite{Becker_Templates_2011} for the low-rank
estimation stage (\ref{eq:pre-estimate}). We also used a variant
of the Alternating Direction Method of Multipliers (ADMM) adapted
from \cite{Yang_Alternating_2011} for the sparse estimation stage
(\ref{eq:estimate}). Figure \ref{fig:NEvsRK} illustrates the variation
of the empirical median of $\frac{1}{\sigma^{2}}\left\Vert \widehat{\mb X}-\mb X^{\star}\right\Vert _{F}^{2}$
(i.e., the normalized squared reconstruction error) versus $k$ and
$r$. The spread around each data point indicates the distribution
of the normalized squared reconstruction error. As can be seen from
the figure, $\frac{1}{\sigma^{2}}\left\Vert \widehat{\mb X}-\mb X^{\star}\right\Vert _{F}^{2}$
grows almost linearly with respect to $r$ and $k$ if one of them
is fixed which is in agreement with the theoretical results.
\begin{figure}
\centering
\noindent\subfloat[$\frac{1}{\sigma^{2}}\left\Vert \widehat{\protect\mb X}-\protect\mb X^{\star}\right\Vert _{F}^{2}$
vs. $r$ for $k\in\left\{ 10,14,18\right\} $ ]{\protect\includegraphics[width=0.75\textwidth]{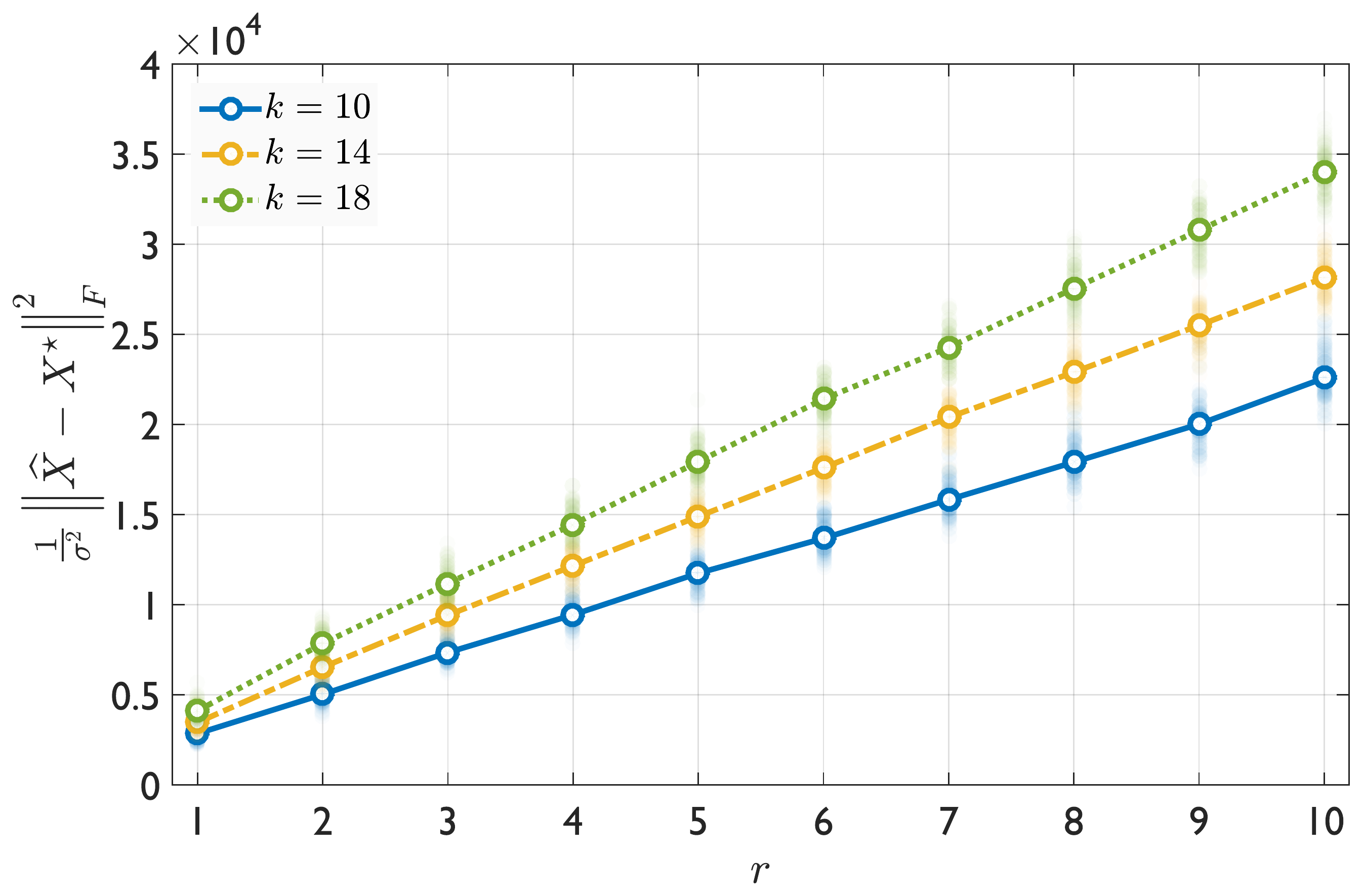}}

\centering
\noindent\subfloat[$\frac{1}{\sigma^{2}}\left\Vert \widehat{\protect\mb X}-\protect\mb X^{\star}\right\Vert _{F}^{2}$
vs. $k$ for $r\in\left\{ 2,6,10\right\} $ ]{\protect\includegraphics[width=0.75\textwidth]{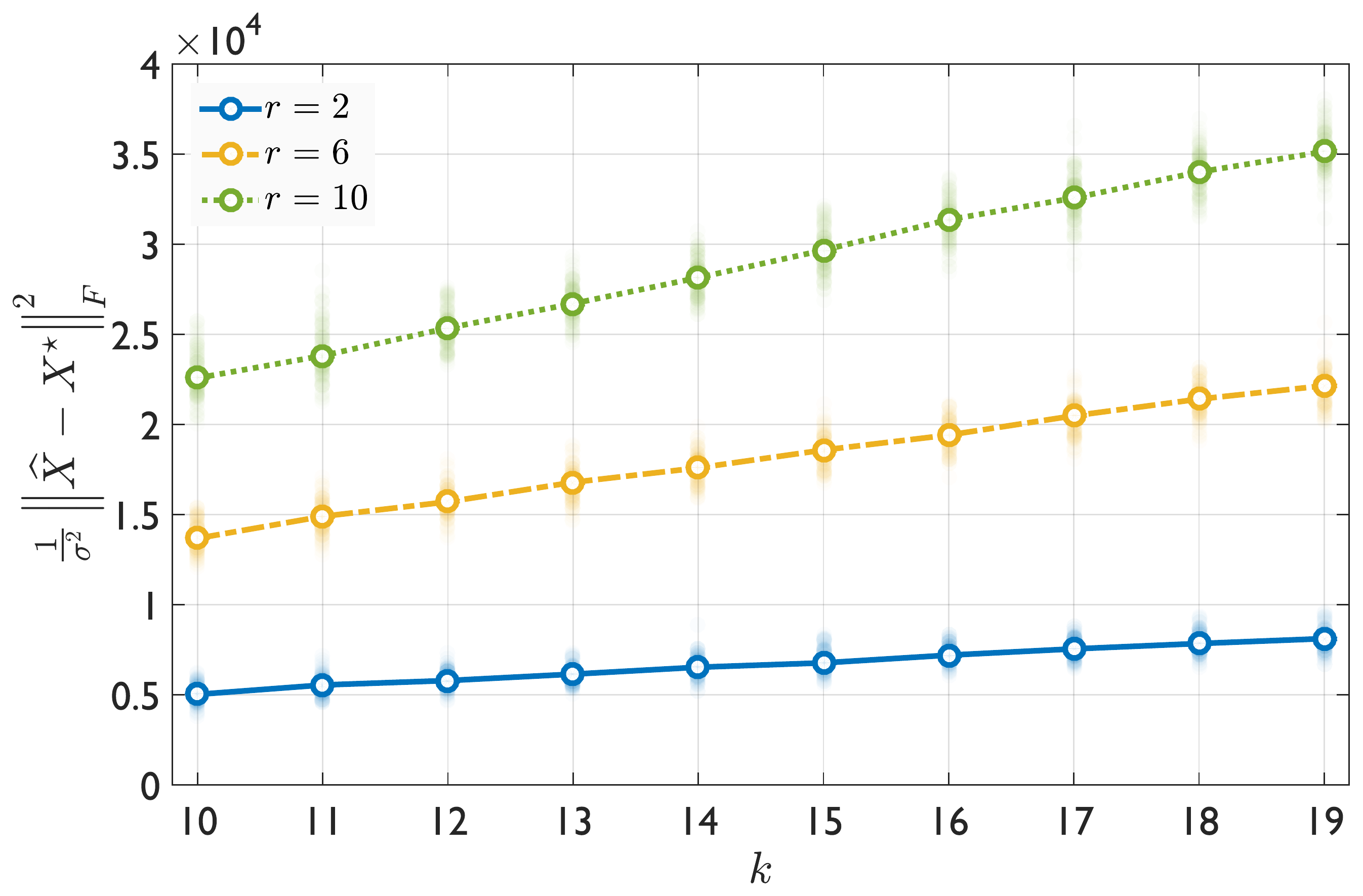}}\protect\caption{\label{fig:NEvsRK}Normalized squared reconstruction error for different
values of rank (i.e., $r$) and row-wise sparsity (i.e., $k$)}
\end{figure}

We also ran a simulation for the CPR described in Section \ref{ssec:CPR} on a two-dimensional image of Saturn.\footnote{Adapted from NASA's Voyager 2 image, 1981-08-24,  NASA catalog \#PIA01364} The target image has  size $1,280\times 1,024$, thus the ambient dimension is $p = 1,280\cdot 1,024= 1,310,720$. The target is \emph{compressible} in the wavelet domain. In particular, with eight levels of two-dimensional 30-tap Coiflet wavelets \cite{Daubechies_Ten_1992} we considered the sparsity of about $k=3,000$. The goal is to estimate the sparse/compressible wavelet coefficients of the target from the phaseless measurements of the form $\mb{y}=\left|\mb{W}\mb{\varPsi}\mb{x}\right|^2 $ where the modulus is taken entrywise. Here $\mb{x}$ represents the vectorized target image, $\mb{\varPsi}$ effectively performs a Rademacher modulation followed by $m=\lceil 2k(1+\log\frac{p}{k})\rceil=42,479$ random Fourier measurement, and $\mb{W}$ makes Fourier measurements at $m$ randomly chosen frequencies for each of $C=20$ masks with \emph{octanary} patterns described in \cite{Candes_Phase_2014}. Therefore, the matrix $\mb{\varPsi}$ has $m=42,479$ rows, and the total number of measurements (i.e., the number of rows in $\mb{W}$) is $n=Cm= 849,580$. We compared the performance of our two-stage approach with the estimate produced through ordinary phase retrieval. In both cases, we used 500 iterations of the Wirtinger-Flow algorithm \cite{Candes_Phase_2014} as the phase retrieval solver. For the sparse recovery stage of our proposed method we relied on $100$ iterations of the IHT algorithm \cite{Blumensath_Iterative_2009}. Figure \ref{fig:CPR} shows, respectively from left to right, the original image, the recovered image using ordinary phase retrieval, and the recovered image through the proposed two-stage CPR. As counting the degrees of freedom would suggest, the ordinary phase retrieval is expected to fail given that we only obtained $n\approx 0.65 p$ measurement. The two-stage recovery, however, produced an estimate with relative error of less than $5\%$. The specified algorithms and measurement schemes allowed us run the simulation without significant memory or computational requirements. Without computational and memory restrictions, the two-stage method can achieve a similar accuracy at even lower sampling rates (i.e., $n/p$)  by using generic (e.g., Gaussian) measurements.
\begin{figure}
\centering
\noindent
\subfloat[original]{\includegraphics[width=0.32\textwidth]{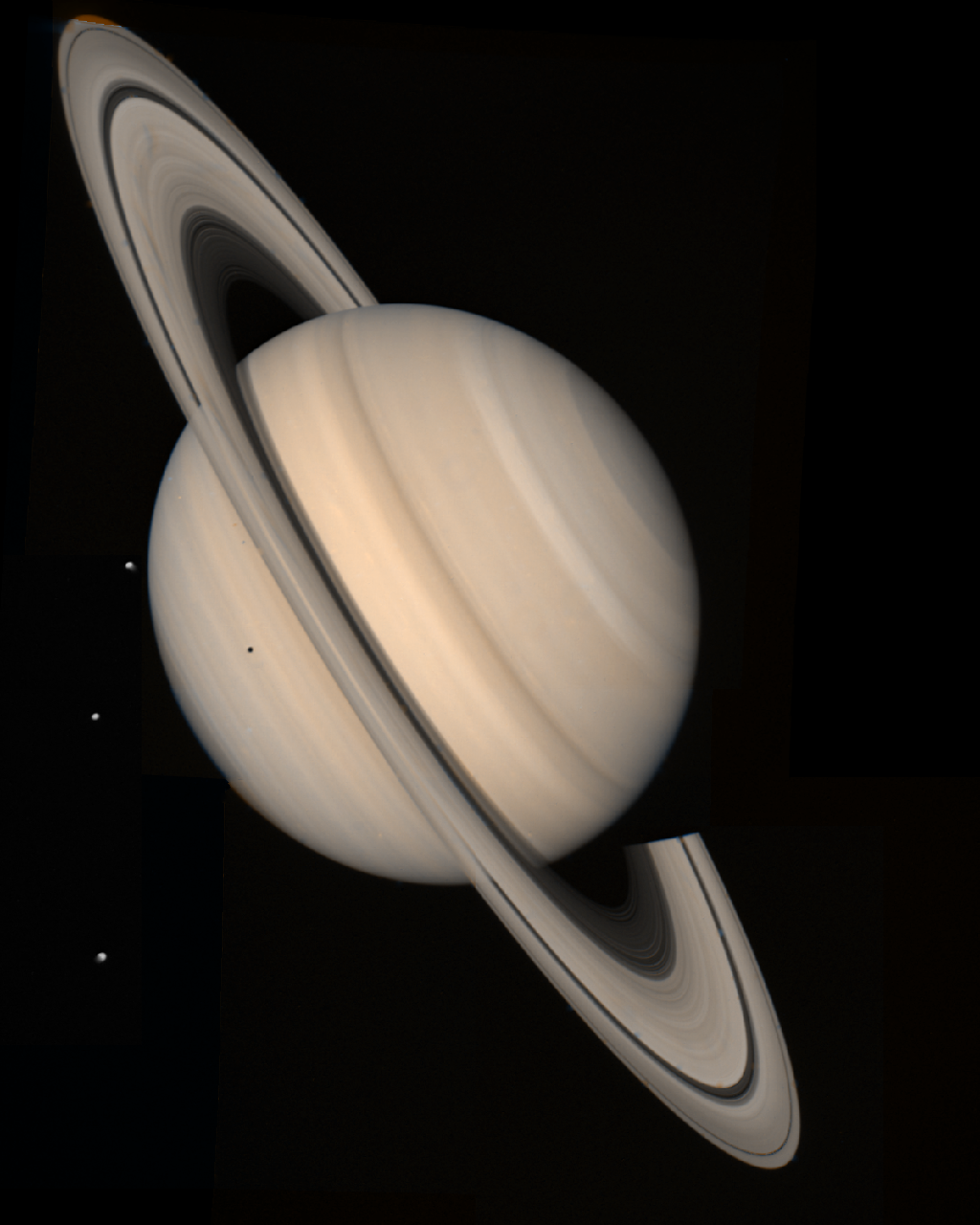}}
\hfill
\subfloat[non-sparse phase retrieval]{\includegraphics[width=0.32\textwidth]{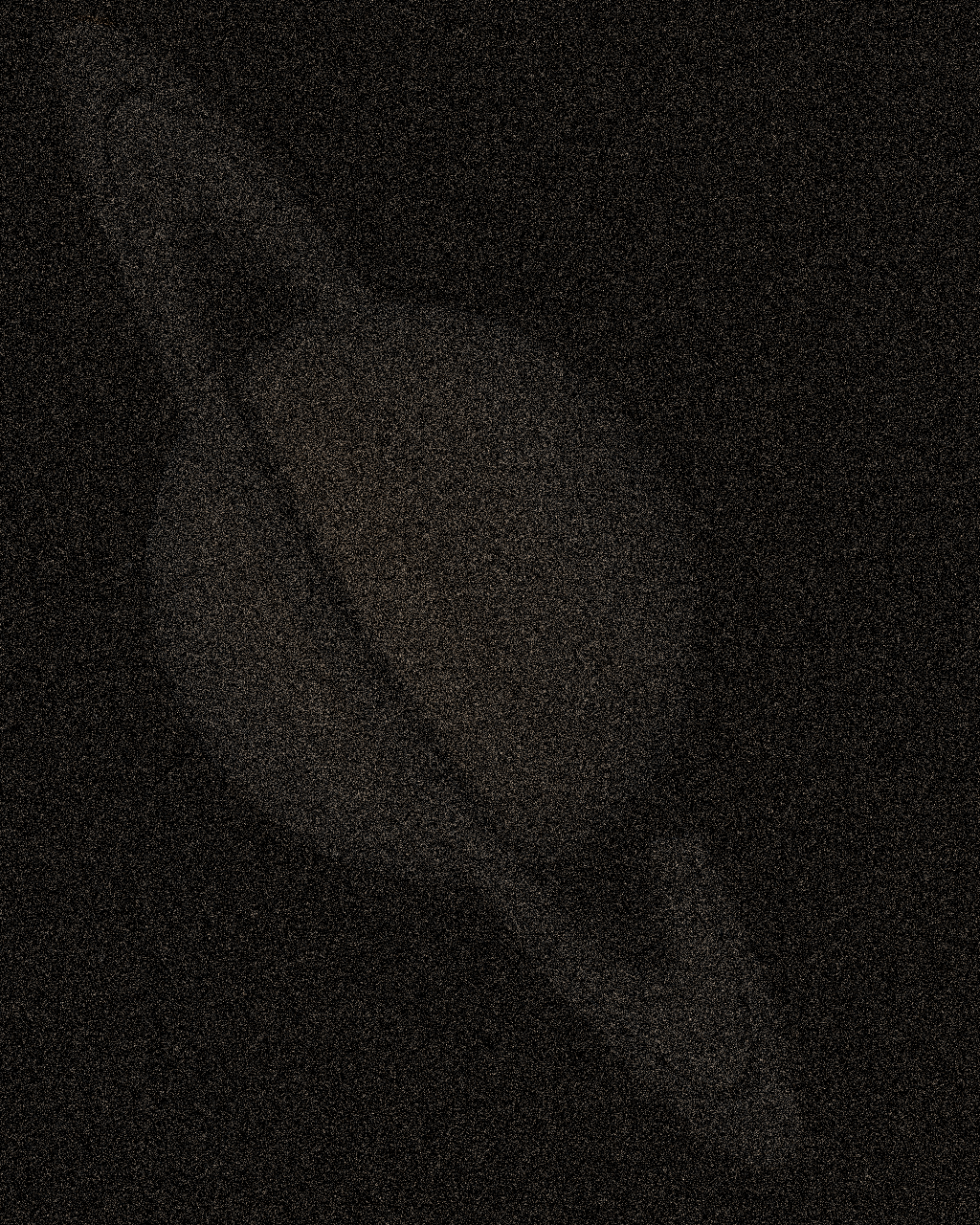}}
\hfill
\subfloat[sparse phase retrieval]{\includegraphics[width=0.32\textwidth]{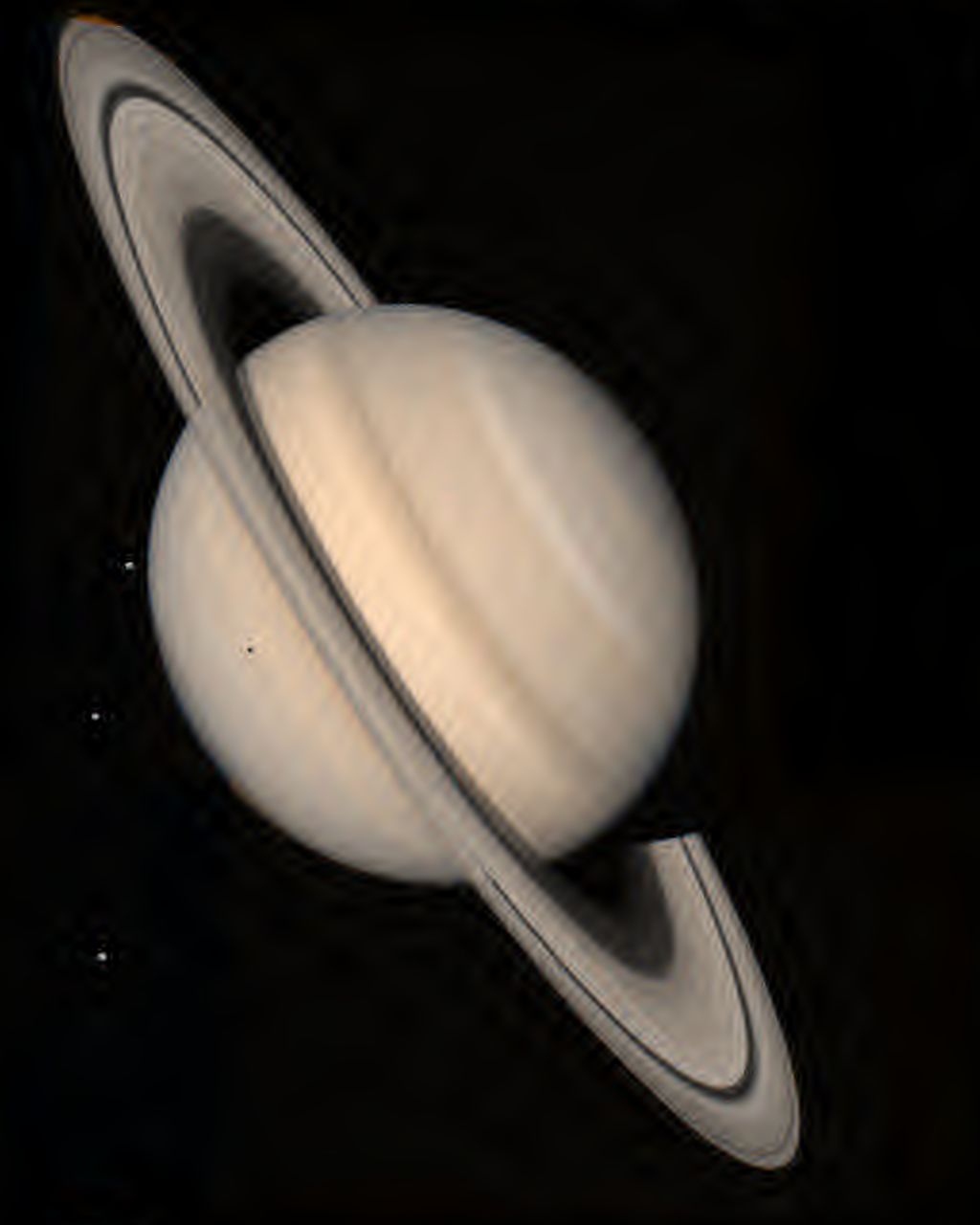}}
\caption{\label{fig:CPR} Compressive phase retrieval of Saturn from nested measurements. While ordinary phase retrieval yields a poor estimate, the proposed two-stage approach that exploits sparsity has produced an estimate with relative error less than $5\%$.}
\end{figure}

\section*{Funding}
This work was supported by ONR grant N00014-11-1-0459 and NSF grant CCF-1422540.

\bibliographystyle{imaiai}
\bibliography{references}

\appendix
\section{Auxiliary tools}

To facilitate readability of the proofs, in this section we state
the most important auxiliary results in our derivations that we borrowed
from the literature.
\begin{thm}[{Fano-type inequality \cite[Theorem 2.5]{Tsybakov_Introduction_2008}}]
\label{thm:Fano}Let $d\left(\cdot,\cdot\right)$ be a metric and
$D\left(\P\parallel\P'\right)$ denote the Kullback-Leibler divergence
of probability measures $\P$ and $\P'$. Assume that $M\geq2$ and
suppose that $\varTheta$ contains elements $\theta_{0},\theta_{1},\dotsc,\theta_{M}$
associated with probability measures $\P_{j}:=\P_{\theta_{j}}$ such
that:
\begin{enumerate}
\item d$\left(\theta_{j},\theta_{k}\right)\geq2s>0,\qquad\forall0\leq j<k\leq M,$
\item $\P_{j}$ is absolutely continuous with respect to $\P_{0}$ for all
$j=1,2,\dotsc,M$, and

\begin{align*}
\frac{1}{M}\sum_{j=1}^{M}D\left(\P_{j}\parallel\P_{0}\right) & \leq\alpha\log M
\end{align*}
 with $0<\alpha<\frac{1}{8}$.

\end{enumerate}
Then 
\begin{align*}
\inf_{\widehat{\theta}}\sup_{\theta\in\varTheta}\P_{\theta}\left(d\left(\widehat{\theta},\theta\right)\geq s\right) & \geq\frac{\sqrt{M}}{1+\sqrt{M}}\left(1-2\alpha-\sqrt{\frac{2\alpha}{\log M}}\right)>0.
\end{align*}
\end{thm}
\begin{lem}[{Varshamov-Gilbert bound variation \cite[Lemma 4.10]{Massart_Concentration_2007}}]
\label{lem:VG} Let $\left\{ 0,1\right\} ^{N}$be equipped with Hamming
distance $d_{\msf H}$ and given $1\leq D<N$ define $\left\{ 0,1\right\} _{D}^{N}=\left\{ \mb x\in\left\{ 0,1\right\} ^{N}\mid d_{\msf H}\left(\mb 0,\mb x\right)=D\right\} $.
For every $\alpha\in\left(0,1\right)$ and $\beta\in\left(0,1\right)$
such that $D\leq\alpha\beta N$, there exists some subset $\varTheta$
of $\left\{ 0,1\right\} _{D}^{N}$ with the following properties 
\begin{align*}
d_{\msf H}\left(\theta,\theta'\right) & >2\left(1-\alpha\right)D\ \forall\left(\theta,\theta'\right)\in\varTheta^{2}\text{ with }\theta\neq\theta',\\
\log\left|\varTheta\right| & \geq\rho D\log\left(\frac{N}{D}\right),
\end{align*}
 where 
\begin{align*}
\rho & =\frac{\alpha}{-\log\left(\alpha\beta\right)}\left(-\log\left(\beta\right)+\beta-1\right).
\end{align*}

\end{lem}

\section{Proofs}

\subsection{Proof of the minimax upper bound}

The following lemma shows the accuracy of the first stage of the estimation
procedure. The proof is analogous to the standard analysis of compressive
sensing and low-rank matrix estimation based on the restricted isometry
property. In particular, our derivations are similar to that of \cite{Candes_Tight_2011}.
We also refer to some of the lemmas proved in \cite{Candes_Tight_2011}
without repeating their statements explicitly. 
\begin{lem}[Low-rank reconstruction]
\label{lem:pre-estimate} Suppose that $n=\msf O\left(r\left(m\vee p_{2}\right)\right)$
and the constant $c_{1}$ in (\ref{eq:pre-estimate}) is sufficiently
large to guarantee that $\mb B^{\star}$ is feasible. If the restricted
isometry constant $\delta_{\tbullet,4r}\left(\mc W\right)<\frac{\sqrt{2}-1}{2}$,
then for some constant $c_{2}>0$ depending only on the restricted
isometry constants, the pre-estimate $\widehat{\mb B}$ obtained from
(\ref{eq:pre-estimate}) obeys 
\begin{align*}
\left\Vert \widehat{\mb B}-\mb B^{\star}\right\Vert _{F} & \leq c_{2}\sigma\sqrt{r\left(m\vee p_{2}\right)},
\end{align*}
 with high probability.\end{lem}
\begin{proof}
It follows from optimality of $\widehat{\mb B}$ and feasibility of
$\mb B^{\star}$ in (\ref{eq:pre-estimate}) that 
\begin{align}
\left\Vert \widehat{\mb B}\right\Vert _{*} & \leq\left\Vert \mb B^{\star}\right\Vert _{*}.\label{eq:optimality}
\end{align}
Let $\mb U\mb{\varSigma}\mb V^{\msf T}$ be the (compact) singular
value decomposition of the rank-$r$ matrix $\mb B^{\star}\in\mbb R^{m\times p_{2}}$.
Denote the subspace of matrices that are ``supported'' on $\mb U$
and $\mb V$ by 
\begin{align*}
T_{0} & =\left\{ \mb B\in\mbb R^{m\times p_{2}}\mid\left(\mb I-\mb U\mb U^{\msf T}\right)\mb B\left(\mb I-\mb V\mb V^{\msf T}\right)=\mb 0\right\} .
\end{align*}
 Furthermore, let $\mb E=\widehat{\mb B}-\mb B^{\star}$ and define
$\mb E_{0}$ to be the projection of $\mb E$ onto the subspace $T_{0}$.
We have 
\begin{align*}
\left\Vert \mb B^{\star}\right\Vert _{*} & \geq\left\Vert \widehat{\mb B}\right\Vert _{*}=\left\Vert \mb B^{\star}+\mb E-\mb E_{0}+\mb E_{0}\right\Vert _{*}\\
 & \geq\left\Vert \mb B^{\star}+\mb E-\mb E_{0}\right\Vert _{*}-\left\Vert \mb E_{0}\right\Vert _{*}\\
 & =\left\Vert \mb B^{\star}\right\Vert _{*}+\left\Vert \mb E-\mb E_{0}\right\Vert _{*}-\left\Vert \mb E_{0}\right\Vert _{*},
\end{align*}
where the first and the second line follow from (\ref{eq:optimality})
and the triangle inequality, and the third line follows from the fact
that $\mb B^{\star}$ and $\mb E-\mb E_{0}$ have mutually orthogonal
columnspaces and rowspaces since $\mb E-\mb E_{0}\in T_{0}^{\perp}$
(see \cite[Lemma 2.3]{Recht_Guaranteed_2010}). Therefore, we deduce
that 
\begin{align}
\left\Vert \mb E-\mb E_{0}\right\Vert _{*} & \leq\left\Vert \mb E_{0}\right\Vert _{*}\leq\sqrt{2r}\left\Vert \mb E_{0}\right\Vert _{F}.\label{eq:nuc-tail}
\end{align}

Using the definition of the measurement $\mb y=\mc W\left(\mb B^{\star}\right)+\mb z$
we can write 
\begin{align*}
\left\Vert \mc W\left(\widehat{\mb B}\right)-\mb y\right\Vert _{2}^{2}=\left\Vert \mc W\left(\mb E\right)-\mb z\right\Vert _{2}^{2} & =\left\Vert \mc W\left(\mb E\right)\right\Vert _{2}^{2}-2\left\langle \mc W\left(\mb E\right),\mb z\right\rangle +\left\Vert \mb z\right\Vert _{2}^{2}\\
 & \geq\left\Vert \mc W\left(\mb E\right)\right\Vert _{2}^{2}-2\left\Vert \mb E\right\Vert _{*}\left\Vert \mc W^{*}\left(\mb z\right)\right\Vert +\left\Vert \mb z\right\Vert _{2}^{2}.
\end{align*}
It is shown in \cite[Lemma 1.1]{Candes_Tight_2011} that for a sufficiently
large constant $c>0$ we have $\left\Vert \mc W^{*}\left(\mb z\right)\right\Vert \leq c\sigma\sqrt{m\vee p_{2}}$
with probability exceeding $1-2e^{-c_{0}n}$ where $c_{0}$ is an
absolute constant. Therefore, given the feasibility of $\widehat{\mb B}$
and (\ref{eq:noise-bound}) we obtain 
\begin{align*}
\left\Vert \mc W\left(\mb E\right)\right\Vert _{2}^{2} & \leq2c\sigma\left\Vert \mb E\right\Vert _{*}\sqrt{m\vee p_{2}}+2c_{1}\sigma^{2}r\left(m\vee p_{2}\right).
\end{align*}
Then, using the triangle inequality $\left\Vert \mb E\right\Vert _{*}\leq\left\Vert \mb E_{0}\right\Vert _{*}+\left\Vert \mb E-\mb E_{0}\right\Vert _{*}$
and (\ref{eq:nuc-tail}) we deduce that
\begin{align}
\left\Vert \mc W\left(\mb E\right)\right\Vert _{2}^{2} & \leq4c\sigma\sqrt{2r}\left\Vert \mb E_{0}\right\Vert _{F}\sqrt{m\vee p_{2}}+2c_{1}\sigma^{2}r\left(m\vee p_{2}\right)\nonumber \\
 & \leq4c\sigma\sqrt{2r}\left\Vert \mb E\right\Vert _{F}\sqrt{m\vee p_{2}}+2c_{1}\sigma^{2}r\left(m\vee p_{2}\right).\label{eq:L(Delta)-upper}
\end{align}
For $i=1,2,\dotsc$, define the matrices $\mb E_{i}$ recursively
as the best rank-$2r$ approximation of $\mb E-\sum_{j=0}^{i-1}\mb E_{j}$
until $\mb E_{i}=\mb 0$. Clearly, we have 
\begin{align*}
\mb E & =\sum_{j\geq0}\mb E_{j},
\end{align*}
and thereby 
\begin{align*}
\mc W\left(\mb E\right) & =\sum_{j\geq0}\mc W\left(\mb E_{j}\right).
\end{align*}
 To simplify the notation we use the shorthand $\delta_{\tbullet,4r}=\delta_{\tbullet,4r}\left(\mc W\right)$
below. It follows from \cite[Lemma 3.3]{Candes_Tight_2011} that for
any pair of distinct indices $j$ and $j'$ we have 
\begin{align*}
\left\langle \mc W\left(\mb E_{j}\right),\mc W\left(\mb E_{j'}\right)\right\rangle  & \geq-\delta_{\tbullet,4r}\left\Vert \mb E_{j}\right\Vert _{F}\left\Vert \mb E_{j'}\right\Vert _{F}.
\end{align*}
 Therefore, we can expand $\left\Vert \mc W\left(\mb E\right)\right\Vert _{2}^{2}$
and write 
\begin{align}
\left\Vert \mc W\left(\mb E\right)\right\Vert _{2}^{2} & =\left\Vert \sum_{j\geq0}\mc W\left(\mb E_{j}\right)\right\Vert _{2}^{2}\nonumber \\
 & =\left\Vert \mc W\left(\mb E_{0}\right)+\mc W\left(\mb E_{1}\right)\right\Vert _{2}^{2}+\left\Vert \sum_{j\geq2}\mc W\left(\mb E_{j}\right)\right\Vert _{2}^{2}+2\sum_{j\geq2}\left\langle \mc W\left(\mb E_{0}\right)+\mc W\left(\mb E_{1}\right),\mc W\left(\mb E_{j}\right)\right\rangle \nonumber \\
 & =\left\Vert \mc W\left(\mb E_{0}\right)+\mc W\left(\mb E_{1}\right)\right\Vert _{2}^{2}+\sum_{j\geq2}\left\Vert \mc W\left(\mb E_{j}\right)\right\Vert _{2}^{2}\nonumber \\
 & +2\sum_{j'>j\geq2}\left\langle \mc W\left(\mb E_{j'}\right),\mc W\left(\mb E_{j}\right)\right\rangle +2\sum_{j\geq2}\left\langle \mc W\left(\mb E_{0}\right)+\mc W\left(\mb E_{1}\right),\mc W\left(\mb E_{j}\right)\right\rangle \nonumber \\
 & \geq\left(1-\delta_{\tbullet,4r}\right)\left(\left\Vert \mb E_{0}+\mb E_{1}\right\Vert _{F}^{2}+\sum_{j\geq2}\left\Vert \mb E_{j}\right\Vert _{F}^{2}\right)\nonumber \\
 & -2\delta_{\tbullet,4r}\sum_{j\geq2}\left(\left\Vert \mb E_{0}\right\Vert _{F}+\left\Vert \mb E_{1}\right\Vert _{F}\right)\left\Vert \mb E_{j}\right\Vert _{F}-2\delta_{\tbullet,4r}\sum_{j'>j\geq2}\left\Vert \mb E_{j}\right\Vert _{F}\left\Vert \mb E_{j'}\right\Vert _{F}\nonumber \\
 & \geq\left\Vert \mb E\right\Vert _{F}^{2}-\delta_{\tbullet,4r}\left(\left\Vert \mb E_{0}+\mb E_{1}\right\Vert _{F}^{2}+2\left(\left\Vert \mb E_{0}\right\Vert _{F}+\left\Vert \mb E_{1}\right\Vert _{F}\right)\sum_{j\geq2}\left\Vert \mb E_{j}\right\Vert _{F}\right)\label{eq:L(Delta)-lower0}\\
 & -\delta_{\tbullet,4r}\left(\sum_{j\geq2}\left\Vert \mb E_{j}\right\Vert _{F}\right)^{2}.\nonumber 
\end{align}
Because $\mb E_{0}$ and $\mb E_{1}$ are orthogonal we have $\left\Vert \mb E_{0}\right\Vert _{F}+\left\Vert \mb E_{1}\right\Vert _{F}\leq\sqrt{2}\left\Vert \mb E_{0}+\mb E_{1}\right\Vert _{F}$.
Furthermore, the construction of $\mb E_{j}$ guarantees that 
\begin{align*}
\left\Vert \mb E_{j}\right\Vert _{F} & \leq\frac{1}{\sqrt{2r}}\left\Vert \mb E_{j-1}\right\Vert _{*},
\end{align*}
 for $j\geq2$. Since for $j\geq1$ both the columnspaces and the
rowspaces of the matrices $\mb E_{j}$ are mutually orthogonal, we
can invoke \cite[Lemma 2.3]{Recht_Guaranteed_2010} and write 
\begin{align*}
\sum_{j\geq2}\left\Vert \mb E_{j-1}\right\Vert _{*}=\left\Vert \sum_{j\geq2}\mb E_{j-1}\right\Vert _{*}=\left\Vert \mb E-\mb E_{0}\right\Vert _{*} & .
\end{align*}
Therefore, in view of (\ref{eq:nuc-tail}) we obtain 
\begin{align*}
\sum_{j\geq2}\left\Vert \mb E_{j}\right\Vert _{F}\leq\frac{1}{\sqrt{2r}}\sum_{j\geq2}\left\Vert \mb E_{j-1}\right\Vert _{*}=\frac{1}{\sqrt{2r}}\left\Vert \mb E-\mb E_{0}\right\Vert _{*} & \leq\left\Vert \mb E_{0}\right\Vert _{F}\leq\left\Vert \mb E_{0}+\mb E_{1}\right\Vert _{F}.
\end{align*}
 Therefore, we can simplify the bound (\ref{eq:L(Delta)-lower0})
to 
\begin{align}
\left\Vert \mc W\left(\mb E\right)\right\Vert _{2}^{2} & \geq\left\Vert \mb E\right\Vert _{F}^{2}-2\left(1+\sqrt{2}\right)\delta_{\tbullet,4r}\left\Vert \mb E_{0}+\mb E_{1}\right\Vert _{F}^{2}\nonumber \\
 & \geq\left(1-2\left(1+\sqrt{2}\right)\delta_{\tbullet,4r}\right)\left\Vert \mb E\right\Vert _{F}^{2}.\label{eq:L(Delta)-lower}
\end{align}
 If $\delta_{\tbullet,4r}<\frac{\sqrt{2}-1}{2}$, then (\ref{eq:L(Delta)-upper})
and (\ref{eq:L(Delta)-lower}) yield 
\begin{align*}
\left\Vert \mb E\right\Vert _{F} & \leq c_{2}\sigma\sqrt{r\left(m\vee p_{2}\right)},
\end{align*}
 for some constant $c_{2}>0$ that only depends on the restricted
isometry constants of $\mc W$.
\end{proof}

We are now ready to prove Theorem \ref{thm:upper-bound}.

\begin{proof}[Proof of Theorem \ref{thm:upper-bound}]
 Recall that $\mb B^{\star}=\mb{\varPsi}\mb X^{\star}$. Lemma \ref{lem:pre-estimate}
guarantees that $\widehat{\mb B}$ obtained from (\ref{eq:pre-estimate})
obeys 
\begin{align*}
\left\Vert \mb{\varPsi}\mb X^{\star}-\widehat{\mb B}\right\Vert _{F}=\left\Vert \mb B^{\star}-\widehat{\mb B}\right\Vert _{F} & \leq c_{2}\sigma\sqrt{r\left(m\vee p_{2}\right)},
\end{align*}
where $c_{2}$ is constant depending only on the restricted isometry
constants of $\mc W$. Therefore, $\mb X^{\star}$ is feasible for
(\ref{eq:estimate}). The fact that $\mb{\varPsi}$ is also a restricted
isometry allows us to apply results from compressive sensing of block-sparse
signals (see, e.g., \cite[Theorem 2 and Section VI]{Eldar_Robust_2009})
and obtain 
\begin{align*}
\left\Vert \widehat{\mb X}-\mb X^{\star}\right\Vert _{F} & \leq C\sigma\sqrt{r\left(m\vee p_{2}\right)},
\end{align*}
 where $C>0$ depends on $\delta_{\tbullet,4r}\left(\mc W\right)$
through $c_{2}$ and on $\delta_{2k,\tbullet}\left(\mb{\varPsi}\right)$.
\end{proof}

\subsection{Proof of the minimax lower bound}

We follow a standard strategy to establish the minimax lower bound
using information theoretic tools. The minimax probability of inaccurate
estimation over $\mbb X_{k,r}$ can be bounded from below by the same
quantity over any subset of $\mbb X_{k,r}$, particularly the finite
subsets. This probability can be further relaxed to the minimax probability
of error in a hypothesis testing problem defined for the same subset.
Therefore, to obtain a tight lower bound on the minimax probability
of error, it suffices to construct a large set of hypotheses that
are difficult to distinguish based on the observations. In particular,
we need to choose a finite but sufficiently large number of potential
target matrices in $\mbb X_{k,r}$ such that they are well-separated
but they are difficult to distinguish from their noisy compressive
measurements. To this end, we construct two different finite subsets
$\mbb X'$ and $\mbb X''$ of the set of $p_{1}\times p_{2}$ matrices
of rank no more than $r$ that have at most $k$ nonzero rows and
Frobenius norm at most $\varepsilon$ which is denoted by 
\begin{align*}
\mbb X{}_{k,r,\varepsilon} & \coloneqq\left\{ \mb X\in\mbb X_{k,r}\mid\left\Vert \mb X\right\Vert _{F}\leq\varepsilon\right\} .
\end{align*}
 Construction of these two subsets depend critically on the choice
of the sets $\mbb S$, $\mathbb{T}'$, and $\mbb T''$ that are subsets
of $\left\{ S\subset\left[p_{1}\right]|\left|S\right|=k\right\} $,
$\left\{ \pm1\right\} ^{r\times p_{2}}$, and $\left\{ \pm1\right\} ^{k\times r}$,
respectively. Elements of the set $\mbb S$ determine the indices
of the nonzero rows, whereas the elements of $\mbb T'$ and $\mbb T''$
determine the value of the nonzero entries. As explained below, we
rely on Lemma \ref{lem:VG}, a variant of the Varshamov-Gilbert bound
established in \cite{Massart_Concentration_2007}, to appropriately
choose $\mbb S$, $\mbb T'$, and $\mbb T''$. Finally, using the
hypothesis sets $\mbb X'$ and $\mbb X''$ we derive a minimax lower
bound by invoking Theorem \ref{thm:Fano} which is one of the Fano-type
inequalities established in \cite{Tsybakov_Introduction_2008}.

\begin{proof}[Proof of Theorem \ref{thm:lower-bound}]
Given $\mbb S$, $\mbb T'$, and $\mbb T''$ define 
\begin{align*}
\mbb X' & \coloneqq\left\{ \frac{\varepsilon}{\sqrt{kp_{2}}}\left[\begin{array}{cc}
\mb I_{p_{1},S} & \mb 0_{r\left\lceil \frac{k}{r}\right\rceil -k}\end{array}\right]\left(\mb 1_{\left\lceil \frac{k}{r}\right\rceil \times1}\otimes\mb T\right)\mid S\in\mbb S\ \text{and}\ \mb T\in\mbb T'\right\} ,
\end{align*}
and 
\begin{align*}
\mbb X'' & \coloneqq\left\{ \frac{\varepsilon}{\sqrt{kp_{2}}}\mb I_{p_{1},S}\left(\mb T\otimes\mb 1_{1\times\left\lceil \frac{p_{2}}{r}\right\rceil }\right)\mb I_{r\left\lceil \frac{p_{2}}{r}\right\rceil ,\left[p_{2}\right]}\mid S\in\mbb S\ \text{and}\ \mb T\in\mbb T''\right\} .
\end{align*}
In words, each matrix in $\mbb X'$ basically consists of $\frac{k}{r}$
copies of an $r\times p_{2}$ binary matrix that are stacked on top
of each other and interleaved with all-zero rows. Clearly, the matrices
in $\mbb X'$ have $k$ nonzero rows, are of rank at most $r$, and
have Frobenius norm of $\varepsilon$, showing that $\mbb X'\subset\mbb X_{k,r,\varepsilon}$.
Similarly, each matrix in $\mbb X''$ is essentially a horizontal
concatenation of $\frac{p_{2}}{r}$ copies of an $k\times r$ binary
matrix that is interleaved row-wise with all-zero rows. It is straightforward
to verify that the matrices in $\mbb X''$ have $k$ nonzero rows,
are of rank at most $r$, and have Frobenius norm of $\varepsilon$,
which show that $\mbb X''\subset\mbb X_{k,r,\varepsilon}$ as well.

We use the Varshamov-Gilbert bound as stated in Lemma \ref{lem:VG}
to choose sufficiently large sets $\mbb S$, $\mbb T'$, and $\mbb T''$.
Treating each of the sets in $\mbb S$ as a binary sequence of length
$p_{1}$, Lemma \ref{lem:VG}  guarantees existence of a set $\mbb S$
of $k$-subsets of $\left[p_{1}\right]$ such that 
\begin{align*}
\log\left|\mbb S\right| & \geq\frac{4}{25}k\log\frac{p_{1}}{k}, & \text{ and } &  & d_{\msf H}\left(S_{1},S_{2}\right) & \geq\frac{1}{4}k,\ \forall\left(S_{1},S_{2}\right)\in\mbb S^{2},\,S_{1}\neq S_{2}.
\end{align*}
 We can also treat each matrix in $\mbb T'$ as a binary string of
length $rp_{2}$, apply Lemma \ref{lem:VG},  and show that there
exists a set $\mbb T'$ of matrices in $\left\{ \pm1\right\} ^{r\times p_{2}}$
that satisfies 
\begin{align*}
\log\left|\mbb T'\right| & \geq\frac{3}{25}rp_{2}, & \text{ and } &  & d_{\msf H}\left(\mb T_{1},\mb T_{2}\right) & \geq\frac{1}{8}rp_{2},\ \forall\left(\mb T_{1},\mb T_{2}\right)\in{\mbb T'}^{2},\,\mb T_{1}\neq\mb T_{2}.
\end{align*}
 Similarly, there exists a set $\mbb T''$ of matrices in $\left\{ \pm1\right\} ^{k\times r}$
such that 
\begin{align*}
\log\left|\mbb T''\right| & \geq\frac{3}{25}rk, & \text{ and } &  & d_{\msf H}\left(\mb T_{1},\mb T_{2}\right) & \geq\frac{1}{8}rk,\ \forall\left(\mb T_{1},\mb T_{2}\right)\in{\mbb T''}^{2},\,\mb T_{1}\neq\mb T_{2}.
\end{align*}

Let $\left(S_{1},\mb T_{1}\right)$ and $\left(S_{2},\mb T_{2}\right)$
be two distinct pairs in $\mbb S\times\mbb T'$ using which we can
construct the matrices $\mb X_{1}$ and $\mb X_{2}$ in $\mbb X'$,
respectively. If $S_{1}\neq S_{2}$, then counting only the rows of
$\mb X_{1}-\mb X_{2}$ that are not in the set $S_{1}\cap S_{2}$
we obtain 
\begin{align*}
\left\Vert \mb X_{1}-\mb X_{2}\right\Vert _{F} & \geq\frac{\varepsilon}{\sqrt{kp_{2}}}\sqrt{d_{\msf H}\left(S_{1},S_{2}\right)p_{2}}\geq\frac{\varepsilon}{2}.
\end{align*}
Furthermore, if $S_{1}=S_{2}$ and $\mb T_{1}\neq\mb T_{2}$, then
we have 
\begin{align*}
\left\Vert \mb X_{1}-\mb X_{2}\right\Vert _{F} & \geq\frac{2\varepsilon}{\sqrt{kp_{2}}}\sqrt{d_{\msf H}\left(\mb T_{1},\mb T_{2}\right)\left\lfloor \frac{k}{r}\right\rfloor }\geq\varepsilon\sqrt{\frac{r\left\lfloor \frac{k}{r}\right\rfloor }{2k}}\geq\frac{\varepsilon}{2}.
\end{align*}
 Therefore, the above inequalities and the fact that $\log\left|\mbb X'\right|=\log\left|\mbb S\right|+\log\left|\mbb T'\right|$
show that the set $\mbb X'$ obeys 
\begin{align}
\log\left|\mbb X'\right| & \geq\frac{4}{25}k\log\frac{p_{1}}{k}+\frac{3}{25}rp_{2}, & \text{ and } &  & \left\Vert \mb X_{1}-\mb X_{2}\right\Vert _{F} & \geq\frac{\varepsilon}{2},\ \forall\left(\mb X_{1},\mb X_{2}\right)\in{\mbb X'}^{2},\,\mb X_{1}\neq\mb X_{2}\label{eq:logA1}
\end{align}
 Similarly, we can show that $\mbb X''$ obeys 
\begin{align}
\log\left|\mbb X''\right| & \geq\frac{4}{25}k\log\frac{p_{1}}{k}+\frac{3}{25}rk, & \text{ and } &  & \left\Vert \mb X_{1}-\mb X_{2}\right\Vert _{F} & \geq\frac{\varepsilon}{2},\ \forall\left(\mb X_{1},\mb X_{2}\right)\in{\mbb X''}^{2},\,\mb X_{1}\neq\mb X_{2}\label{eq:logA2}
\end{align}

For any matrix $\mb X\in\mbb R^{p_{1}\times p_{2}}$ let $\P_{\mb X}$
denote the Gaussian distribution $\msf N\left(\mc A\left(\mb X\right),\sigma^{2}\mb I\right)$
which is the distribution of the measurement $\mb y$ if $\mb X$
is the target matrix. Recall that by assumption $\mc A$ is a restricted
isometry over $\mbb X_{k,r}$ as defined by (\ref{eq:A-bounded})
with the restricted isometry constant $\delta=\delta_{k,r}\left(\mc A\right)$.
For any $\mb X\in\mbb X'\subset\mbb X_{k,r}$ the KL-divergence between
$\P_{\mb X}$ and $\P_{\mb 0}$ can be bounded as 
\begin{align*}
D\left(\P_{\mb X}\|\P_{\mb 0}\right) & =\frac{\left\Vert \mc A\left(\mb X\right)\right\Vert _{2}^{2}}{2\sigma^{2}}\leq\frac{\gamma_{k,r}\left\Vert \mb X\right\Vert _{F}^{2}}{2\sigma^{2}}=\frac{\gamma_{k,r}\varepsilon^{2}}{2\sigma^{2}},
\end{align*}
 where the inequality follows from restricted isometry assumption
on $\mc A$. Therefore, we have 
\begin{align*}
\frac{1}{\left|\mbb X'\right|}\sum_{\mb X\in\mbb X'}D\left(\P_{\mb X}\|\P_{\mb 0}\right) & \leq\frac{\gamma_{k,r}\varepsilon^{2}}{2\sigma^{2}}.
\end{align*}
Suppose that we have 
\begin{align*}
\frac{\gamma_{k,r}\varepsilon^{2}}{2\sigma^{2}} & \leq\alpha'\log\left|\mbb X'\right|,
\end{align*}
holds for some $\alpha'\in\left(0,\frac{1}{8}\right)$. Then, because
any matrix $\mb X\in\mbb X'$ also satisfies $\left\Vert \mb X-\mb 0\right\Vert _{F}=\left\Vert \mb X\right\Vert _{F}=\varepsilon>\frac{\varepsilon}{2}$,
we can use (\ref{eq:logA1}) and invoke Theorem \ref{thm:Fano} to
guarantee that 
\begin{align*}
\inf_{\widehat{\mb X}}\sup_{\mb X^{\star}\in\mbb X_{k,r,\varepsilon}}\P\left(\left\Vert \widehat{\mb X}-\mb X^{\star}\right\Vert _{F}\geq\frac{\varepsilon}{4}\right) & \geq\frac{\sqrt{\left|\mbb X'\right|}}{1+\sqrt{\left|\mbb X'\right|}}\left(1-2\alpha'-\sqrt{\frac{2\alpha'}{\log\left|\mbb X'\right|}}\right).
\end{align*}
From (\ref{eq:logA1}) we have the crude lower bound $\log\left|\mbb X'\right|\geq\frac{3}{25}$.
Therefore, with $\alpha'=5\times10^{-5}$ we can choose 
\begin{align*}
\varepsilon & =10^{-2}\sigma\sqrt{\frac{k\log\frac{p_{1}}{k}+rp_{2}}{\gamma_{k,r}}},
\end{align*}
 and obtain 
\begin{align}
\inf_{\widehat{\mb X}}\sup_{\mb X^{\star}\in\mbb X_{k,r,\varepsilon}}\P\left(\left\Vert \widehat{\mb X}-\mb X^{\star}\right\Vert _{F}\geq2.5\times10^{-3}\sigma\sqrt{\frac{k\log\frac{p_{1}}{k}+rp_{2}}{\gamma_{k,r}}}\right) & >\frac{1}{2}.\label{eq:row-partition}
\end{align}
 Furthermore, for the set $\mbb X''$ if we have 
\begin{align*}
\frac{\left(1+\delta\right)\varepsilon^{2}}{2\sigma^{2}} & \leq\alpha''\log\left|\mbb X''\right|,
\end{align*}
for some $\alpha''\in\left(0,\frac{1}{8}\right)$, then using (\ref{eq:logA2})
we can apply Theorem \ref{thm:Fano} and show that 
\begin{align*}
\inf_{\widehat{\mb X}}\sup_{\mb X^{\star}\in\mbb X_{k,r,\varepsilon}}\P\left(\left\Vert \widehat{\mb X}-\mb X^{\star}\right\Vert _{F}\geq\frac{\varepsilon}{4}\right) & \geq\frac{\sqrt{\left|\mbb X''\right|}}{1+\sqrt{\left|\mbb X''\right|}}\left(1-2\alpha''-\sqrt{\frac{2\alpha''}{\log\left|\mbb X''\right|}}\right).
\end{align*}
Similar to the argument for $\mbb X'$, with $\alpha''=5\times10^{-5}$
we can choose 
\begin{align*}
\varepsilon & =10^{-2}\sigma\sqrt{\frac{k\log\frac{p_{1}}{k}+rk}{\gamma_{k,r}}},
\end{align*}
 and obtain 
\begin{align}
\inf_{\widehat{\mb X}}\sup_{\mb X^{\star}\in\mbb X_{k,r,\varepsilon}}\P\left(\left\Vert \widehat{\mb X}-\mb X^{\star}\right\Vert _{F}\geq2.5\times10^{-3}\sigma\sqrt{\frac{k\log\frac{p_{1}}{k}+rk}{\gamma_{k,r}}}\right) & >\frac{1}{2}.\label{eq:col-partition}
\end{align}
Then with $\varepsilon=2.5\times10^{-3}\sigma\sqrt{\frac{k\log\frac{p_{1}}{k}+r\left(k\vee p_{2}\right)}{\gamma_{k,r}}}$
we can deduce from (\ref{eq:row-partition}) and (\ref{eq:col-partition})
that 
\begin{align*}
\inf_{\widehat{\mb X}}\sup_{\mb X^{\star}\in\mbb X_{k,r,\varepsilon}}\P\left(\left\Vert \widehat{\mb X}-\mb X^{\star}\right\Vert _{F}\geq2.5\times10^{-3}\sigma\sqrt{\frac{k\log\frac{p_{1}}{k}+r\left(k\vee p_{2}\right)}{\gamma_{k,r}}}\right) & >\frac{1}{2},
\end{align*}
which also guarantees the desired minimax lower bound on $\mbb X_{k,r}$. \end{proof}

\end{document}